\documentclass[12pt]{article}
\usepackage{amsmath, amsthm, amssymb,enumerate}%,algorithm2e,tikz,caption,bbm}
\usepackage{fullpage}
\usepackage{enumitem}
\usepackage{hyperref}
\title{Shotgun reconstruction in the hypercube}
\author{Micha{\l} Przykucki\thanks{School of Mathematics, University of Birmingham, Edgbaston, Birmingham, United Kingdom. Supported by the EPSRC grant EP/P026729/1. \newline  E-mail: \texttt{m.j.przykucki@bham.ac.uk}.} \and Alexander Roberts\thanks{Mathematical Institute, University of Oxford, Andrew Wiles Building, Radcliffe Observatory Quarter, Woodstock Road, Oxford, United Kingdom. \newline  E-mail: \texttt{\{robertsa, scott\}@maths.ox.ac.uk}.} \and Alex Scott\footnotemark[2] \thanks{Supported by a Leverhulme Trust Research Fellowship.}}

\newtheoremstyle{case}{}{}{\normalfont}{}{shape}{:}{ }{}

\newtheorem{thm}{Theorem}[section]
\newtheorem{lem}[thm]{Lemma}
\newtheorem{prop}[thm]{Proposition}

\newtheorem{cor}[thm]{Corollary}

\theoremstyle{definition}
\newtheorem{defn}[thm]{Definition}

\newtheorem{qu}[thm]{Question}

\numberwithin{equation}{section}

\newtheoremstyle{case}{}{}{\normalfont}{}{shape}{\normalfont:}{ }{}

\theoremstyle{case}

\def\comment#1{}

\newcommand{\beq}{\begin{eqnarray*}}
\newcommand{\eeq}{\end{eqnarray*}}

\def\build#1_#2^#3{\mathrel{\mathop{\kern 0pt#1}\limits_{#2}^{#3}}}                 

\newcommand{\bP}{\mathbb{P}}

\newcommand{\bN}{\mathbb{N}}

\newcommand{\cP}{\mathcal{P}}

\numberwithin{equation}{section}

\begin{document}

\maketitle

\begin{abstract}
\noindent
Mossel and Ross raised the question of when a random colouring of a graph can be reconstructed from local information, namely the colourings (with multiplicity) of balls of given radius.   In this paper, we are concerned with random $2$-colourings of the vertices of the $n$-dimensional hypercube, or equivalently random Boolean functions.  In the worst case, balls of diameter $\Omega(n)$ are required to reconstruct. However, the situation for {random} colourings is dramatically different: we show that almost every $2$-colouring can be reconstructed from the multiset of colourings  of balls of radius $2$.  Furthermore, we show that for $q \ge n^{2+\epsilon}$, almost every $q$-colouring can be reconstructed from the multiset of colourings of $1$-balls.
\end{abstract}

\section{Introduction}
The problem of reconstructing a graph from a collection of its subgraphs goes back to the famous \emph{reconstruction conjecture} of Kelly and Ulam (see \cite{K1,U1,H1}), which asserts that every graph $G$ on at least 3 vertices can be determined up to isomorphism from the multiset of its vertex-deleted subgraphs, i.e.~the graphs $G-v$ for all $v \in V(G)$. The conjecture has been confirmed for various classes of graphs, including trees, regular graphs and triangulations
(see Nash-Williams \cite{NashWill}, Bondy \cite{bondy} and Lauri and Scapellato \cite{LScap}).  
There has also been a substantial amount of work on the problem of reconstructing a graph, or some other combinatorial structure, from objects of smaller size (see for example, Alon, Caro, Krasikov and Roditty \cite{ACKR}, Pebody, Radcliffe and Scott \cite{PRS}, and Simon \cite{JanSim}).  

Recently, Mossel and Ross \cite{MR} investigated the problem of reconstructing a graph using {\em local} information.  Given a graph, when is it possible to reconstruct the graph up to isomorphism from the multiset of balls of radius $r$?  For graphs in which the vertices or edges are coloured (not necessarily properly), when is it possible to reconstruct the coloured graph from the multiset of coloured $r$-balls?  Motivated by the problems of reconstructing DNA sequences  from ``shotgunned'' stretches of the sequence, as well as neural networks from local subnetworks, they called this type of problem {\em shotgun reconstruction}.   

Mossel and Ross were particularly interested in reconstruction problems where the graph or colouring is {\em random}.  
Reconstructing random objects usually requires much less information than reconstructing in the worst case (see, for example, Bollob\'as \cite{bolnum3} and Radcliffe and Scott \cite{radsco}). Mossel and Ross \cite{MR} proved results on reconstructing sparse random graphs in the $\mathcal G(n,p)$ model, while Mossel and Sun \cite{MS} proved rather sharp bounds on the smallest radius $r$ needed to reconstruct random regular graphs.   Mossel and Ross also considered the problem of reconstructing randomly coloured trees, randomly coloured lattices in any fixed number of dimensions, and the \emph{random jigsaw puzzle problem}, in which the edges of the  $n \times n$ square lattice are randomly coloured with $q$ colours, and the problem is to determine for which $q$ it is possible to reconstruct the original jigsaw from the collection of $1$-balls.  The random jigsaw puzzle problem has since been studied by Bordenave, Feige and Mossel \cite{BFM}, Nenadov, Pfister, and Steger \cite{NPS}, Balister, Bollob\'{a}s, and Narayanan \cite{BBN}, and by Martinsson \cite{M1}.
 
In this paper, we will be interested in shotgun assembly for
vertex-colourings of the $n$-dimensional hypercube $Q_n$. We begin by discussing 2-colourings, or equivalently Boolean functions. In the worst case, it is easy to see that balls of radius at least $n/2-O(1)$ are necessary (consider the two colourings where all points are in colour $1$, except for two points at Hamming distance either $n$ or $n-1$ which have colour $2$).  However for random colourings the situation is dramatically different.  As we shall see, it is not hard to show that for a random $2$-colouring, balls of radius $3$ are almost surely enough for reconstruction, while balls of radius $1$ are not.  The first main result of this paper is that balls of radius $2$ are sufficient.
 
 \begin{thm}
 Almost every $2$-colouring of the hypercube $Q_n$ is reconstructible from the multiset of its coloured $2$-balls.
 \end{thm}
 
In fact, we prove a stronger result (Theorem \ref{main}), which allows imbalanced colourings in which one colour can have density as low at $n^{-1/4+o(1)}$.

  We also consider colourings with more than two colours.  In our other main result, we show that that for sufficiently large $q$ a random $q$-colouring can be reconstructed from its $1$-balls (see Theorem \ref{1ball} for a slightly stronger statement of this result). 
  
 \begin{thm}
Let $\epsilon>0$.  For $q\ge n^{2+\epsilon}$, almost every $q$-colouring of $Q_n$ is reconstructible from the multiset of its coloured 1-balls.
 \end{thm}
 
It is easy to show that $\Omega(n)$ colours are required for reconstructability, and it would be interesting to narrow the gap (see Sections \ref{defres} and \ref{waffle} for further discussion).

The rest of the paper is organized as follows.  The remainder of this section contains definitions, as well as more formal statements of our results.
In Section \ref{probability} we prove some probabilistic tools we will use in our proofs. In Section \ref{structure} we prove an isoperimetric result as well as some other structural results regarding subgraphs of the hypercube. In Sections \ref{combination} and \ref{1ballsec} we prove our main theorems, and we conclude the paper in Section \ref{waffle} with some discussion and open questions. 
 
 \subsection{Definitions and results}\label{defres}
 
For all positive integers $n$, we define the \em $n$-dimensional hypercube \em $Q_n = (V,E)$ where $V = \{0,1\}^n$ and $uv \in E$ if the two vertices differ in exactly one co-ordinate. This graph can also be thought of as a graph on the power set of $[n]$, $\cP(n) = \{A \subseteq [n]\}$, where two sets $A,B$ are adjacent if they differ in exactly one element. For a vertex $u \in V$, we inductively let $\Gamma^0(u) = \{u\}$ and $\Gamma^{k}(u) = \bigcup_{v \in \Gamma^{k-1}(u)}\Gamma(v) \setminus \bigcup_{l <k}\Gamma^{l}(u)$ (so $\Gamma^k(v)$ is the set of vertices which have shortest path length exactly $k$ to $v$). We will call $\Gamma^k(v)$ the \em $k$-th neighbourhood \em of $v$. For a subset of the vertices $A \subset V$, we also write $\Gamma(A) = \bigcup_{v \in A} \Gamma(v)$. With the natural understanding of a distance function, we define the \em $r$-ball \em $B_r(v)$ around a vertex $v$ as the subgraph induced by the vertices at distance at most $r$ from $v$ (so for example $B_2(v)$ is induced by $\{v\}\cup \Gamma(v) \cup \Gamma^2(v)$).

We will need some notions of distances between colourings. Suppose $\chi$ and $\lambda$ are $\left\{0,1\right\}$-colourings of the same graph $G = (V,E)$, then we define
	\begin{align*}
		D(\chi,\lambda) = |\left\{w \in V : \chi(w) \neq \lambda(w)\right\}|.
	\end{align*}
For isomorphic graphs $G$ and $H$, and for a colouring $\chi$ of $G$ and $\lambda$ of $H$, we define
	\begin{align*}
		d(\chi,\lambda) = \min_{\text{iso } f:G \rightarrow H} D(\chi,\lambda \circ f),
	\end{align*}
where the minimum is taken over all graph isomorphisms.

We say that two colourings $\chi$ and $\lambda$ on $G$ are \em equivalent \em ($\chi \cong \lambda$) if and only if $d(\chi,\lambda) = 0$, and we define the equivalence class $[\chi]$ of a colouring $\chi$ accordingly ($\left[\chi\right] = \left\{\lambda : \chi \cong \lambda\right\}$). For a colouring $\chi$ and $r\ge 0$, let $\chi^{(r)}(v) := \chi|_{B_r(v)}$ be the coloured $r$--ball around $v$. We say that $\chi$ and $\lambda$ are \em $r$-locally equivalent \em ($\chi \cong_r \lambda$) if and only if there exists a bijection $f: V(G) \rightarrow V(G)$ such that $\chi^{(r)}(v) \cong  \lambda^{(r)}(f(v))$ for all $v \in V$. 

We say that a colouring $\chi$ is \em $r$-distinguishable \em if there is no colouring $\lambda$ such that $\chi \cong_r \lambda$ but $\chi \not\cong \lambda$, and we say $\chi$ is \em $r$-indistinguishable \em if it is not $r$-distinguishable. Thus $\chi$ is $r$-distinguishable if the collection of local colourings of $r$--balls determines the global colouring. Given $r$-locally equivalent colourings $\chi$ and $\lambda$ of the vertices of the hypercube, there exists a bijection $f$ such that $\chi^{(r)}(v) \cong  \lambda^{(r)}(f(v))$ for all $v \in V(Q_n)$. It is clear then that $\lambda = \chi \circ f^{-1}$, and that $\lambda \cong \chi$ if and only if $f$ can be chosen to be a graph isomorphism. In what follows, we define $\chi_f$ by $\chi_f(v) := \chi \circ f^{-1}(v)$. For a colouring $\chi$ of the hypercube $Q_n$ let $\mathrm{Isom}^{(r)}(\chi)$ be the set of bijections $f : V(Q_n) \to V(Q_n)$ such that $\chi^{(r)}(v) \cong (\chi_f)^{(r)}(f(v))$ for all $v \in V(Q_n)$. So $\chi$ is $r$-indistinguishable if and only if there exists a bijection $f \in \mathrm{Isom}^{(r)}(\chi)$ which is not a graph automorphism. In other words, if $\chi$ is $r$-indistinguishable then there exists a bijection $f \in \mathrm{Isom}^{(r)}(\chi)$ and two non-adjacent vertices $u,v \in V(Q_n)$ such that $f(u)f(v) \in E(Q_n)$.

We will concern ourselves with the problem of whether random colourings of the hypercube are distinguishable.

\begin{defn}
Let $\mu$ be a probability mass function on $\bN$. A random \em $\mu$-colouring \em of the hypercube $V(Q_n)$ is an independent collection of random variables $(\chi(v))_{v \in V(Q_n)}$ each with distribution  $\mu$. For a natural number $q$, we will write \em $q$-colouring \em instead of $\mathrm{Unif}([q])$-colouring.
\end{defn}

We show that for $r=2$ and $p$ not too small, with high probability, a random $(p,1-p)$-colouring of the hypercube is $2$-distinguishable.

\begin{thm}\label{main}
Let $\varepsilon > 0$ and let $p = p(n) \in (0,1/2]$ be a function on the natural numbers such that for sufficiently large $n$, $p \ge n^{-1/4 + \varepsilon}$. Let $\chi$ be a random $(p,1-p)$-colouring of the hypercube $Q_n$. Then with high probability, $\chi$ is $2$-distinguishable.
\end{thm}

We are also able to prove a similar result when we have more colours (see Section \ref{waffle} for a discussion of this). A direct corollary of this result is that random colourings of the hypercube are reconstructible with high probability from its $r$-balls for $r \ge 3$. In this case, however, it is not hard to prove a stronger result.

\begin{thm}\label{soft}
Let $\varepsilon > 0$ and let $p = p(n) \in (0,1/2]$ be a function on the natural numbers such that $\tfrac{np}{\log n} \rightarrow \infty$ as $n \rightarrow \infty$. Let $\chi$ be a random $(p,1-p)$-colouring of the hypercube $Q_n$. Then with high probability, $\chi$ is $3$-distinguishable.
\end{thm}

However, Theorem \ref{main} does not extend to $1$--balls. Indeed, if the hypercube is $q$-coloured where $q = o(n)$, then there are asymptotically fewer collections of colourings of the $2^n$ $1$--balls than there are $q$-colourings of the hypercube: let $q(n) = \tfrac{n}{w(n)}$ where $w(n) \rightarrow \infty$ as $n \rightarrow \infty$. Allowing for automorphisms, there are at least $\frac{q^{2^n}}{2^n n!} = 2^{2^n \log(q) (1+o(1))}$ possible colourings of the hypercube; on the other hand there are $q {n+q - 1 \choose q-1}$ ways of colouring a $1$--ball (up to isomorphism). But
\[
{n+q-1 \choose q-1} \le \left(\frac{3n}{q}\right)^q = \left(3w (n)\right)^{\frac{n}{w(n)}} = 2^{n \frac{\log 3w(n)}{w(n)}} = 2^{o(n)},
\]
and so $q {n+q - 1 \choose q-1} = o(2^n)$. Therefore the number of possible collections of colourings of the $1$--balls (assuming $q>2$) is at most
	\begin{align*}
		{2^n + q{n+q \choose q-1} -1 \choose 2^n} &\le {2^n (1+o(1)) \choose 2^n} \nonumber \\
		&\le 2^{2^n \left(1+o(1)\right)} \nonumber \\
		&= o\left(2^{2^n \log(q) \left(1+o(1)\right)}\right). \nonumber
	\end{align*}
Therefore at least $\Omega(n)$ colours are required. For the problem of reconstructing a colouring from the collection of $1$-balls, we prove the following upper bound.

\begin{thm}\label{1ball}
There exists some constant $K>0$ such that the following holds. Let $q \ge n^{2+K\log^{-\frac{1}{2}} n}$ and let $\chi$ be a random $q$-colouring of the hypercube $Q_n$. Then with high probability, $\chi$ is $1$-distinguishable.
\end{thm}

The proof of Theorem \ref{main} has some probabilistic elements but also uses some structural properties of the hypercube. We will need the following stability result for Harper's Theorem for sets of size $n$.

\begin{thm}\label{f-1-stability}
Let $s(n)$ be a function with $s(n) \rightarrow \infty$ and $s(n) = o(n)$ as $n \rightarrow \infty$. Then there exists a constant $C$ (which may depend on $s(n)$) such that the following holds: If $A \subseteq V(Q_n)$ with $|A| = n$ and $|\Gamma(A)| \le \binom{n}{2} + ns(n)$, there exists some $w \in V(Q_n)$ for which $|\Gamma(w) \cap A| \ge n - Cs(n)$.
\end{thm}

Two of the authors have generalised this result to sets of size $\binom{n}{k}$ for a range of $k$ using different techniques \cite[Theorem 1.2]{VISH-MPAR}. Since the proof for $k=1$ is much simpler, we present it here. We remark that Keevash and Long \cite{PKEL} have independently proven a similar result.

Before continuing, let us give a very brief sketch of our approach to Theorems \ref{main} and \ref{1ball}. A colouring of the hypercube $\chi$ is $2$-indistinguishable if there is another colouring $\lambda$ which is not a rotation of $\chi$ but has the same collection of $2$--ball colourings. Recall that we may express $\lambda$ as $\lambda = \chi_f$ where $f$ is a bijection on the hypercube which is not an automorphism. Recall that we write $\mathrm{Isom}^{(2)}(\chi)$ for the collection of bijections $f$ for which $\chi$ and $\chi_f$ have the same collection of $2$--ball colourings. We prove Theorem \ref{main} by showing that with high probability every bijection in $\mathrm{Isom}^{(2)}(\chi)$ is an automorphism, and so no such $\lambda$ can exist.

To do this, we first consider what sort of properties a function $f \in \mathrm{Isom}^{(2)}(\chi)$ would almost surely need to display. In Section \ref{probability} we look at the neighbourhood $\Gamma(v)$ of a vertex $v$, and consider how spread out its image $f^{-1}(\Gamma(v))$ is in the hypercube. We show that with high probability, for every vertex $v$, the second neighbourhood $\Gamma^2(f^{-1}(\Gamma(v)))$ is not very large. From here we prove in Section \ref{structure} that $f^{-1}(\Gamma(v))$ must closely resemble a neighbourhood of a vertex $g(v)$ for each vertex $v$. It follows that with high probability, for each bijection $f \in \mathrm{Isom}^{(2)}(\chi),$ the inverse $f^{-1}$ roughly maps neighbourhoods to neighbourhoods.

This rough mapping of neighbourhoods forces a certain amount of rigidity of $f^{-1}$; around each vertex, there must be a large structure which is invariant under $f^{-1}$. If an $f \in \mathrm{Isom}^{(2)}(\chi)$ exists which is not an automorphism, then there must be two non adjacent vertices $u$ and $v$ with $f^{-1}(u)$ and $f^{-1}(v)$ adjacent. But $u$ and $v$ each have a large structure around them invariant under $f^{-1}$. The colourings of these two large structures must then fit together. We show that the probability of this occurring is small. We may conclude that $\mathrm{Isom}^{(2)}(\chi)$ contains only automorphisms with high probability.

The proof of Theorem \ref{1ball} is similar. This time, we show that with high probability, for every vertex $v$, the neighbourhood $\Gamma(f^{-1}(\Gamma(v)))$ is not very large. Since $q$ is so large, with high probability, the colourings of $1$--balls have very little overlap, and so it cannot be that $f(\Gamma(v))$ has large clusters around more than one vertex. We combine these to show that $f(\Gamma(v))$ has a large cluster around some vertex $g(v)$ for each vertex $v$. The remainder of the proof mimics that of Theorem \ref{main}.

\subsection{Notation}

We record here for reference some notation that will be used later in the proofs. The reader may choose to skip some of these for now, as they will all be introduced in the sections to come.

\begin{itemize}
\item For $i \in [n]$, we define $e_i \in \{0,1\}^n$ as the vector whose $i$-th entry is $1$ and whose other entries are $0$.
\item Given a colouring $\chi$, we write $\chi^{(r)}(v)$ for the restriction of $\chi$ to the $r$--ball around $v$.
\item $\mathrm{Bij}$ is the set of bijections $f : V(Q_n) \rightarrow V(Q_n)$.
\item Given a colouring $\chi$ and a bijection $f \in \mathrm{Bij}$, we define $\chi_f$ by $\chi_f(v):= \chi\left(f^{-1}(v)\right)$.
\item Given a colouring $\chi$, we define $\mathrm{Isom}^{(r)}(\chi) := \left\{f \in \mathrm{Bij} :\chi^{(r)}(v) \cong \chi_f^{(r)}(f(v)), \forall v \in V(Q_n)\right\}.$
\item $\mathrm{Cluster}^r_R = \left\{f \in \mathrm{Bij} : \forall v \in V(Q_n), |\Gamma^r(f(\Gamma(v)))| \le \binom{n}{r+1} + R\right\}$ (see Definition \ref{clusterdef}).
\item $\mathrm{Mono}_s^t$ is the set of bijections $f \in \mathrm{Cluster}^1_s$ for which, for all $v \in V(Q_n)$, there exists at most one vertex $w \in V(Q_n)$ such that $|f(\Gamma(v))\cap \Gamma(w)| > t$ (see Definition \ref{monodef}).
\item $\mathrm{Local}_s$ is the set of $s$-approximately local bijections (see Definition \ref{approxdefn}).
\item $\mathrm{Diag}_s:= \left\{f \in \mathrm{Local}_s : f_{\star \star} = f\right\}$ is the set of diagonal $s$-approximately local bijections (see Definition \ref{diagdefn}).
\item $\mathrm{Self}_s:= \left\{f \in \mathrm{Local}_s : f_{\star} = f\right\}$ is the set of $s$-approximately local bijections for which the dual of $f$ is itself (see Definition \ref{selfdual}).
\end{itemize}

\section{Probabilistic arguments}\label{probability}
In this section we show that we need only consider bijections $f$ such that $f^{-1}$ ``behaves well" on neighbourhoods: for every vertex $v$, the second neighbourhood of $\left\{f^{-1}(w) : w \in \Gamma(v)\right\}$ is not too large. Before we do this, we show that under the assumptions of Theorems \ref{main} and \ref{1ball}, the colourings of $2$--balls and $1$--balls respectively differ greatly from one another. To do this, we will need the following bounds on the tail of the Binomial distribution (see \cite{Chernoffcite} for the proof of Lemma \ref{chernn}).

\begin{lem}[Chernoff's Inequality]\label{chernn}
Let $n \in \bN, p \in (0,1)$ and $\varepsilon>0$. Then
	\begin{align*}
		\bP\left[\mathrm{Bin}(n,p) \le np(1-\varepsilon)\right] \le \exp\left\{-\frac{\varepsilon^2 np}{2}\right\}.
	\end{align*}
\end{lem}

\begin{lem}\label{bounds}
Fix $K > 0$ and let $p =p(n) \in (0,1/2]$ be such that $np \rightarrow \infty$. Then for $0\le c \le K$ such that $n/2 + c\sqrt{n\log n }$ is an integer we have
	\begin{align}
		\bP\left[\mathrm{Bin}(n,p) = np + c\sqrt{np\log np }\right] = \Theta\left((np)^{-\left(\frac{1}{2}+\frac{c^2}{2(1-p)}\right)}\right) \label{bin1}
	\end{align}
uniformly over $c$. Furthermore
	\begin{align}
		\bP\left[\mathrm{Bin}(n,p) \ge np + c\sqrt{np\log np }\right] = \Omega\left((np)^{\frac{1}{3}}\bP\left[\mathrm{Bin}(n,p) = np + c\sqrt{np\log np }\right]\right). \label{bin2}
	\end{align}
\end{lem}
\begin{proof}
Let $K>0$ and suppose $0\le c \le K$. Let $r =c\sqrt{np\log np }$. We first prove \eqref{bin1}. We have
	\begin{align*}
		\bP\left[\mathrm{Bin}(n,p) = np + c\sqrt{np\log np }\right] &= \binom{n}{np+r} p^{np+r} (1-p)^{n(1-p)-r} \\
		&= \frac{n!p^{np+r}(1-p)^{n(1-p)-r}}{(np+r)!(n(1-p)-r)!} \\
		&= \Theta \left ( \frac{\sqrt{n}(n/e)^n p^{np+r} (1-p)^{n(1-p)-r}}{\sqrt{np} \left (\frac{np+r}{e} \right )^{np+r}\sqrt{n(1-p)} \left (\frac{n(1-p)-r}{e} \right )^{n(1-p)-r}} \right ) \\
      		&= \Theta \left ( \frac{1}{\sqrt{np}} \left ( \frac{p}{p+r/n} \right )^{np+r} \left (\frac{1-p}{1-p -r/n} \right )^{n(1-p)-r} \right ) \\
		&= \Theta \left ( \frac{1}{\sqrt{np}} \left ( 1+\frac{r}{np} \right )^{-np-r}\left (1-\frac{r}{n(1-p)} \right )^{-n(1-p)+r} \right).
	\end{align*}
By Taylor expansion of $\log(1+x)$,
\[
\left( 1+\frac{r}{np} \right )^{-np-r} = \exp\left\{-r - \frac{r^2}{2np} + O\left(\frac{r^3}{(np)^2}\right)\right\}.
\]
Analogously,
\[
\left(1-\frac{r}{n(1-p)} \right )^{-n(1-p)+r} = \exp\left\{r-\frac{r^2}{2n(1-p)} + O\left(\frac{r^3}{(n(1-p))^2}\right).\right\}
\]
Therefore
	\begin{align*}
      		&\bP\left[\mathrm{Bin}(n,p) = np + c\sqrt{np\log np }\right] \\
		&= \Theta \left ( \frac{1}{\sqrt{np}}\exp\left\{-r^2\left(\frac{1}{2np} + \frac{1}{2n(1-p)}\right) + O\left(\frac{r^3}{(np)^2}\right)\right\} \right ) \\
      		&= \Theta \left( \frac{1}{\sqrt{np}}\exp\left\{-c^2\log np \left(\frac{1}{2} + \frac{np}{2n(1-p)}\right) + O\left(K^3(np)^{-1/2}\log^{3/2} np \right)\right\} \right ) \\
      		&= \Theta\left((np)^{-\left(\frac{1}{2}+\frac{c^2}{2(1-p)})\right)}\right).
	\end{align*}
 Now \eqref{bin2} follows immediately by observing that for $0 \le t \le n^\frac{1}{3}$,
	\begin{align}
		 \bP\left[\mathrm{Bin}(n,p) = np + c\sqrt{np\log np } + t\right] &\ge  \bP\left[\mathrm{Bin}(n,p) = np + c\sqrt{np\log np } + (np)^{\frac{1}{3}}\right] \nonumber \\
		 &= \Theta\left(\bP\left[\mathrm{Bin}(n,p) = np + c\sqrt{np\log np }\right]\right). \nonumber
	\end{align}
\end{proof}

The next two lemmas show that with high probability the pairwise distances between colourings of the $2$-balls around vertices are large.

\begin{lem}\label{strongunique}
Let $p=p(n) \in (0,1/2]$ be such that $\frac{pn}{\log n} \rightarrow \infty$. Let $\chi$ be a random $(p,1-p)$-colouring of the hypercube $Q_n$. Then with high probability, there do not exist distinct vertices $u,v \in V(Q_n)$ such that $d(\chi^{(2)}(u),\chi^{(2)}(v)) \le \frac{n^2p(1-p)}{2}$.
\end{lem}

\begin{proof}
Let $\chi$ be a random $(p,1-p)$-colouring of the hypercube $Q_n$. Let $u,v \in V(Q_n)$ be distinct vertices and let $b: B_2(u) \rightarrow B_2(v)$ be an isomorphism. Let $T = (B_2(u)\cap B_2(v)) \cup b^{-1}(B_2(u)\cap B_2(v))$ and let $Y = (B_2(u)\setminus T) \cup b(B_2(v)\setminus T)$. Let 
	\begin{align}
		N = \left|\left\{w \in B_2(u)\setminus T : \chi(w) \neq (\chi \circ b) (w)\right\}\right|.
	\end{align}

Since $(\chi(w))_{w \in Y}$ is a collection of independent $(p,1-p)$ random variables,
	\begin{align*}
		N \sim \mathrm{Bin}\left(\frac{n^2+n+2}{2}-|T|,2p(1-p)\right). 
	\end{align*}

A simple counting argument shows that $|T| \le 4n$ and so (for sufficiently large $n$) we may apply Lemma \ref{chernn} to get
	\begin{align*}
		\bP\left[N \le \frac{n^2p(1-p)}{2}\right] &\le \bP\left[\mathrm{Bin}\left(\frac{n^2-8n}{2},2p(1-p)\right)\le \frac{n^2p(1-p)}{2}\right] \\
		&\le \bP\left[\mathrm{Bin}\left(\frac{n^2}{3},2p(1-p)\right) \le \frac{2n^2p(1-p)}{3}\left(1- \frac{1}{4}\right)\right] \\
		&\le \exp\left\{-\frac{n^2p(1-p)}{48}\right\}.
	\end{align*}
Taking a union bound over all possible choices of vertices $u,v$ and isomorphisms $b$ we obtain that the probability that there are distinct vertices $u,v$ with $d(\chi^{(2)}(u),\chi^{(2)}(v)) \le \frac{n^2p(1-p)}{2}$ is at most
	\begin{align*}
		2^{2n}n! \exp\left\{-\frac{n^2p(1-p)}{48}\right\} = o(1).
	\end{align*}
\end{proof}

\begin{lem}\label{1strongunique}
For every $\varepsilon > 0$ there exists a constant $K>0$ such that the following holds: Let $q \ge n^{1+ \varepsilon}$ and let $\chi$ be a random $q$-colouring of the hypercube $Q_n$. Then with high probability, there do not exist distinct vertices $u,v \in V(Q_n)$ such that $d(\chi^{(1)}(u),\chi^{(1)}(v)) \le n - \frac{nK}{\log n }$.
\end{lem}

\begin{proof}
Let $\varepsilon>0$ and let $K > 4/\varepsilon$. Let $q \ge n^{1+ \varepsilon}$ and let $\chi$ be a random $q$-colouring of the hypercube $Q_n$. Let $u,v \in V(Q_n)$ be distinct vertices. Let $T = \Gamma(u) \cap \Gamma(v)$, and let $Y = \Gamma(u) \setminus T$. Then $(\chi(w))_{w \in Y}$ is a collection of independent $\mathrm{Unif}\left([q]\right)$ random variables independent of $S := \left\{\chi(w):w \in \Gamma(v)\right\}$. Let us first observe $S$ and then set $N:= \left\{w \in \Gamma(u) : \chi(w) \in S\right\}$. Then (conditional on $S$) the probability that an arbitrary $r$-tuple of $Y$ is a subset of $N$ is $(|S|/q)^r$. Let $r = \left\lceil \tfrac{nK}{\log n} \right\rceil - 2$. We can apply a union bound to get
	 \begin{align*}
	 	\bP\left[|N| \ge r +2\right] &\le \sum_{Z \in Y^{(r)}}\bP\left[Z \subset N\right\} \le {|Y| \choose r} \left(|S|/q\right)^r \le \left(\frac{e|Y||S|}{rq}\right)^r
	\end{align*}
Since $|S|,|Y| \le n$, for sufficiently large $n$ we therefore have
	\begin{align*}
		\bP\left[|N| \ge r +2\right] \le \left(en^2/rq\right)^r \le \left(3K^{-1} n^{-\varepsilon}\log n\right)^r \le n^{-2\varepsilon r /3} \le 2^{-\frac{\varepsilon K n}{2}}. 
	\end{align*}
Taking a union bound over all possible pairs of distinct vertices $u,v$ we obtain that the probability that there exist distinct vertices $u,v$ with $d(\chi^{(1)}(u),\chi^{(1)}(v)) \le n-\frac{nK}{\log n }$ is at most $2^{2n -\frac{\varepsilon K n}{2}}.$ Since $K > \frac{4}{\varepsilon}$, we see the probability is $o(1)$.
\end{proof}

With the proofs of these lemmas in mind, there is an easy argument proving Theorem \ref{soft}

\begin{proof}[Sketch proof of Theorem \ref{soft}]
Following the proof of Lemma \ref{strongunique}, one can show that the colouring of $2$-balls are unique when $\tfrac{pn}{\log n} \rightarrow \infty.$ Let $\left(\lambda_{\ell}\right)_{\ell \in [2^n]}$ be the collection of colourings of $3$-balls. Without loss of generality, suppose that $\lambda_{1}$ is the colouring of the $3$-ball around $0$. Note then that for each $i \in [n]$, the colouring of the $2$-ball around $e_i$ is contained in $\lambda_{1}$. Since the colourings of $2$-balls are unique, we can then discern which $\ell \in [2^n]$ correspond to neighbours of $0$. We are then iteratively able to work out $B_k(0)$ for $k=1,\ldots,n$.
\end{proof}

We now come to considering the local behaviour of bijections of the hypercube. For this we will need a notion for how spread out the image of a neighbourhood is. Note that if $h$ is an isomorphism then, for any vertex $v$, $|\Gamma^r(h(\Gamma(v)))| = |\Gamma^r(\Gamma(h(v)))| = {n \choose r+1}$.

\begin{defn}\label{clusterdef}
For natural numbers $r$ and $R$ (where $R$ may be a function of $n$) define $\mathrm{Cluster}^r_R$ to be the set of bijections $h : V(Q_n) \rightarrow V(Q_n)$ such that $|\Gamma^r(h(\Gamma(v)))| \le \binom{n}{r+1} + R$ for all $v \in V(Q_n),$ i.e.
	\begin{align*}
		\mathrm{Cluster}^r_R = \left\{h\in \mathrm{Bij} : \forall v \in V(Q_n), |\Gamma^r(h(\Gamma(v)))| \le \binom{n}{r+1} + R\right\}.
	\end{align*}
\end{defn}

We now show that if $\chi$ is a random $2$-colouring and $K>0$ is sufficiently large, then with high probability, every $f \in \mathrm{Isom}^{(2)}(\chi)$ satisfies $f^{-1} \in \mathrm{Cluster}^2_{Kn^2p^{-1}\log n }$. This means that in Theorem \ref{main} we need only consider bijections $f$ such that for every vertex $v$, the set $f^{-1}(\Gamma(v))$ has a second neighbourhood that is close to minimal in size.

\begin{lem}\label{scott}
Let $p=p(n) \in (0,1/2]$ be such that $\frac{np}{\log n} \rightarrow \infty$. Then there exists a constant $K > 0$ such that the following holds: Let $\chi$ be a random $(p,1-p)$-colouring of the hypercube $Q_n$. Then with high probability, every $f \in \mathrm{Isom}^{(2)}(\chi)$ satisfies $f^{-1} \in \mathrm{Cluster}^2_{Kn^2 p^{-1}\log n }$.
\end{lem}

The proof of Lemma \ref{scott} is a little involved so we provide a brief outline here. Let $\chi$ be a random $(p,1-p)$-colouring of the hypercube $Q_n$, let $f \in \mathrm{Isom}^{(2)}(\chi)$ and fix a vertex $v \in V(Q_n)$. Recall that $f \in \mathrm{Isom}^{(2)}(\chi)$ means that $\chi^{(2)}(f^{-1}(w)) \cong \chi_f^{(2)}(w)$ for each neighbour $w$ of $v$. Therefore it is possible to ``match up" $(\chi(u))_{u \in \Gamma^2(f^{-1}(\Gamma(v)))}$ with $(\chi_f(u))_{u \in B_3(v)}$. We bound the probability that this is possible by considering whether it is possible for $(\chi(u))_{u \in \Gamma^2(f^{-1}(\Gamma(v)))}$ to match up with any colouring of $B_3(v)$. If $\Gamma^2(f^{-1}(\Gamma(v)))$ is too large, then this happens with very small probability because we have to match up too many colours. Applying a union bound, we are able to conclude that $\Gamma^2(h^{-1}(\Gamma(x)))$ must be sufficiently small for any $h \in \mathrm{Isom}^{(2)}(\chi)$ and $x \in V(Q_n)$.

\begin{proof}
Let $\chi$ be a random $(p,1-p)$-colouring of the hypercube $Q_n$. Suppose there exists an $f \in \mathrm{Isom}^{(2)}(\chi)$ such that $f^{-1} \not\in \mathrm{Cluster}^2_{Kn^2 p^{-1}\log n }$ (for $K>0$ to be determined later). Pick $v \in V(Q_n)$ such that $|\Gamma^2(f^{-1}(\Gamma(v)))| > \binom{n}{3} + Kn^2 p^{-1}\log n$. Since $f \in \mathrm{Isom}^{(2)}(\chi)$,
\[
\chi_f^{(2)}(v+e_i) \cong \chi ^{(2)}(f^{-1}(v+e_i))
\]
for each $i \in [n]$, where we carry out addition mod $2$. Thus, there is a permutation $\pi^i$ of $[n]$ such that for all distinct $j,k \in [n]$
	\begin{align*}
		\chi_f(v+e_i + e_j + e_k) = \chi(f^{-1}(v+e_i) + e_{\pi^i(j)}+e_{\pi^i(k)}).
	\end{align*}
	
Let $A = \left\{f^{-1}(v+e_1),\ldots,f^{-1}(v+e_n)\right\}$, so then $(\chi(u))_{u \in \Gamma^2(A)}$ is determined by $(\chi \circ f^{-1}(u))_{u \in \Gamma(v)\cup \Gamma^3(v)}$ and $(\pi^i)_{i \in n}$. Therefore there must exist a $2$-colouring $c$ of $\Gamma(v)\cup \Gamma^3(v)$, a subset $A \subset V(Q_n)$ for which $|A| = n$ and $\Gamma^2(A) > \binom{n}{3} + Kn^2 p^{-1}\log n $, and a family of permutations $(\pi^i)_{i\in [n]}$, which is compatible with $(\chi(u))_{u \in \Gamma^2(A)}$. Fix a vertex $v$, a colouring $c$, a set $A$ and a family of permutations $(\pi^i)_{i\in [n]}$. 

We may express each vertex $w \in \Gamma(v)\cup \Gamma^3(v)$ as $w = v+ e_i + e_j + e_k$ where $j \neq k$. Further fix this expression for $w$ so that $i$ is as small as possible and $j < k$ (so for each $w \in \Gamma(v)\cup \Gamma^3(v)$ we have fixed $i,j,k$ such that $w = v+e_i + e_j + e_k$). Then if the vertex $v$, the colouring $c$, the set $A$ and the family of permutations $(\pi^i)_{i\in [n]}$ are compatible with $(\chi(u))_{u \in \Gamma^2(v)}$, we have $\chi(f^{-1}(v+e_i) + e_{\pi^i(j)}+e_{\pi^i(k)}) = c(w)$. For ease of reading, define $h$ by $$h(i,j,k) := f^{-1}(v+e_i) + e_{\pi^i(j)}+e_{\pi^i(k)}.$$ By independence, the probability that $(\chi(u))_{u \in \Gamma^2(A)}$ is compatible with $v, c, A$ and $(\pi^i)_{i \in [n]}$ is
	\begin{align}
		\prod_{h(i,j,k) \in \Gamma^2(A)} p^{1-c(v+e_i+e_j+e_k)}(1-p)^{c(v+e_i+e_j+e_k)}. \label{messyprod}
	\end{align}
(Note that we are using the colours $0$ and $1$.)

We have an injection $t : \Gamma(v)\cup \Gamma^3(v) \rightarrow \Gamma^2(A)$ such that $\chi \circ t = c$. Let $B = t(\Gamma(v) \cup \Gamma^3(v))$. Splitting \eqref{messyprod} into $B$ and $\Gamma^2(A) \setminus B$ gives

	\begin{align}
		&\prod_{h(i,j,k) \in \Gamma^2(A)\setminus B} p^{1-c(v+e_i+e_j+e_k)}(1-p)^{c(v+e_i+e_j+e_k)} \prod_{x \in B} p^{1-c(t^{-1}(x))}(1-p)^{c(t^{-1}(x))} \nonumber \\
		&\le (1-p)^{|\Gamma^2(A) \setminus B|}\prod_{w \in \Gamma(v) \cup \Gamma^3(v)} p^{1-c(w)}(1-p)^{c(w)}. \nonumber
	\end{align}
	
The right hand product is the probability that a random $(p,1-p)$-colouring of $\Gamma(v) \cup \Gamma^3(v)$ (denote this random colouring $Q$) is equal to $c$. Recall that $|\Gamma^2(A)| \ge \binom{n}{3} + Kn^2 p^{-1}\log n$ and $|B| = \binom{n}{3} + n$ so that $|\Gamma^2(A) \setminus B| \ge Kn^2 p^{-1}\log n - n$. Therefore the probability that $(\chi(u))_{u \in \Gamma^2(A)}$ is compatible with $v, c, A$ and $(\pi^i)_{i \in [n]}$ is at most
	\begin{align}
		(1-p)^{|\Gamma^2(A) \setminus B|} \bP\left[Q = c\right] &\le \exp\left\{-p(Kn^2 p^{-1}\log n -n)\right\}\bP\left[Q = c\right] \nonumber \\
		&\le \exp\left\{-\frac{K}{2}n^2\log n\right\}\bP\left[Q = c\right].
	\end{align}
The number of choices for $v,A$ and the permutations $(\pi^i)_{i\in [n]}$ is at most
	\begin{align}
		2^n 2^{n^2} (n!)^n &\le \exp\left\{Cn^2\log n \right\}. \label{picky1}
	\end{align}
So the probability that $(\chi(u))_{u\in \Gamma^2(A)}$ is compatible with a fixed $c$ and any such choice of $v,A$ and permutations $(\pi^i)_{i\in [n]}$ is at most
	\begin{align}
		\exp\left\{Cn^2\log n\right\}\exp\left\{-\frac{K}{2}n^2\log n\right\}\bP\left[Q = c\right] = \exp\left\{\left(C-\frac{K}{2}\right)n^2\log n \right\} \bP\left[Q=c\right]. \nonumber
	\end{align}
	
Finally, we sum over the colourings to get that the probability $(\chi(u))_{u\in \Gamma^2(A)}$ is compatible for any such choice of $v,c,A$ and permutations is at most $\exp\left\{(C-\frac{K}{2})n^2\log n \right\}$. This upper bound is $o(1)$ provided $K> 2C$.
\end{proof}

In fact for any $C>1$, if $n$ is sufficiently large, then \eqref{picky1} holds, and so the result holds for any $K>2$. A similar result holds for $q$-colourings of $1$--balls.

\begin{lem}\label{1scott}
Let $\alpha > 0$ and let $\varepsilon : \bN \rightarrow [\alpha,\infty)$. Then there exists a constant $K > 0$ such that the following holds: Let $q \ge Kn^{1+\frac{1}{2\varepsilon(n)}}$ and let $\chi$ be a random $q$-colouring of the hypercube $Q_n$. Then with high probability, every $f \in \mathrm{Isom}^{(1)}(\chi)$ satisfies $f^{-1} \in \mathrm{Cluster}^1_{\varepsilon(n) n^2}$.
\end{lem}

In Theorem \ref{1ball}, we consider $q = n^{2 + \Theta(\log^{-\frac{1}{2}} n)}$ which corresponds to $\varepsilon(n) = \frac{1}{2} - \Theta(\log^{-\frac{1}{2}}(n))$. The proof of Lemma \ref{1scott} is much like the the proof of Lemma \ref{scott} but in order to minimise the exponent in Theorem \ref{1ball}, we carefully bound the choice of permutations.

\begin{proof}[Proof of Lemma \ref{1scott}]
Let $\varepsilon = \varepsilon(n)$ be as above. Let $K>0$ be a constant (which we will choose later) and let $q \ge Kn^{1+\frac{1}{2\varepsilon}}$. Let $\chi$ be a random $q$-colouring of the hypercube $Q_n$. Suppose there exists an $f \in \mathrm{Isom}^{(1)}(\chi)$ such that $f^{-1} \not\in \mathrm{Cluster}^1_{\varepsilon n^2}$, and pick $v \in V(Q_n)$ such that $|\Gamma(f^{-1}(\Gamma(v)))| > \binom{n}{2} + \varepsilon n^2$. Note that for each $i \in [n]$, $\chi_f ^{(1)}(v+e_i) \cong \chi ^{(1)}(f^{-1}(v+e_i))$ and so there are permutations $\pi^i$ of $[n]$ for each $i \in [n]$, such that for distinct $i,j \in [n]$
	\begin{align}
		\chi_f(v+e_i + e_j) = \chi(f^{-1}(v+e_i) + e_{\pi^i(j)}). \nonumber
	\end{align}
Let $A = \left\{f^{-1}(v+e_1),\ldots,f^{-1}(v+e_n)\right\}$. Then $(\chi(u))_{u \in \Gamma(A)}$ is determined by $(\chi_f(u))_{u \in B_2(v)}$ and $(\pi^i)_{i \in n}$. Therefore there must exist a $q$-colouring $c$ of $B_2(v)$, a subset $A \subset V(Q_n)$ for which $|A| = n$ and $\Gamma(A) > \binom{n}{2} + \varepsilon n^2$, and a family of permutations $(\pi^i)_{i\in [n]}$, which determines $(\chi(u))_{u \in \Gamma(A)}$. Fix a vertex $v$, a colouring $c$, a set $A$, and a family of permutations $(\pi^i)_{i\in [n]}$. Then the probability that $(\chi(u))_{u \in \Gamma(A)}$ is compatible with $v, c, A$ and $(\pi^i)_{i \in [n]}$ is $q^{-|\Gamma(A)|}$.

There are $2^n$ choices for $v$, and $q^{\binom{n}{2} + O(n)}$ choices for the colouring $c$, and at most $2^{n^2}$ choices for the set $A$. Fix a vertex $v$, a colouring $c$ and fix $A= \left\{f^{-1}(v+e_1),\ldots,f^{-1}(v+e_n)\right\}$ with $|\Gamma(A)| \ge \binom{n}{2}+ 1 + \varepsilon(n) n^2$. For ease of reading we define $a_i = f^{-1}(v+e_i)$ for each $i \in [n]$. Since $c^{(1)}(v+e_{i}) = \chi^{(1)}(a_i)$, there has to exist a permutation $\pi^i$ such that $c(v+e_i+e_k) = \chi(a_i + e_{\pi^i(k)})$ for all $k \in [n]$. For each $i \in [n]$, consider an equivalence relation $\sim_i$ on permutations where $\pi \sim_i \pi'$ if and only if $c(v+e_i+e_{\pi(k)}) = c(v+e_i+e_{\pi'(k)})$ for all $k \in [n]$. For each $i\in [n]$, pick an arbitrary permutation from each equivalence class to form a set of representatives $P^i$. So then for all $i\in [n]$ there must be a $\pi^i \in P^i$ such that $c(v+e_i+e_k) = \chi(a_i + e_{\pi^i(k)})$ for all $k \in [n]$.

Let $r_i = |\Gamma(a_i) \setminus \Gamma(\left\{a_1,\ldots,a_{i-1}\right\})|$. Note that if we have picked permutations $\pi^1,\ldots,\pi^{i-1}$, then we have at most $r_i!$ choices from $P^i$ for permutation $\pi^i$ (since the colours of $n-r_i$ neighbours of $a_i$ have already been determined). We can therefore bound the total number of choices for the permutations (from the $P^i$) by
	\begin{align*}
		\prod_{i\in[n]}r_i! \le n^{\sum_{i\in[n]}r_i} = n^{|\Gamma(A)|}.
	\end{align*}
By a union bound, the probability that $(\chi(u))_{u \in \Gamma(A)}$ is compatible with any choice of $v, c, A$ and $(\pi^i)_{i \in [n]}$ is at most
	\begin{align}
		2^n q^{\binom{n}{2} + O(n)}  2^{n^2} n^{|\Gamma(A)|}q^{-|\Gamma(A)|} &\le q^{\binom{n}{2} + O(n)}2^{O(n^2)}(n/q)^{\binom{n}{2} + \varepsilon n^2} \nonumber \\
		&\le n^{n^2(\frac{1}{2} + \varepsilon)}q^{-\varepsilon n^2 + O(n)}2^{O(n^2)}. \nonumber
	\end{align}

Recalling that $q \ge Kn^{1+\frac{1}{2\varepsilon}}$ and that $\varepsilon \ge \alpha$ we see that this probability is at most
	\begin{align}
		n^{n^2(\frac{1}{2} + \varepsilon)}n^{-\varepsilon n^2 - \frac{1}{2}n^2 + O(n)}K^{-\varepsilon n^2 + O(n)}2^{O(n^2)} \le 2^{O(n^2)}K^{-\alpha n^2}. \nonumber
	\end{align}
If $K$ is sufficiently large, then this upper bound is $o(1)$ and we are done.
\end{proof}

\section{Structural results}\label{structure}
Let $A \subseteq V(Q_n)$ with $|A| = n$. In this section, we start by proving a stability result regarding the size of the neighbourhood of $A$. We will also prove a slightly weaker stability result when the neighbourhood of $A$ is allowed to be quite large. This allows us to later deduce some properties of functions $f \in \mathrm{Isom}^{(r)}(\chi)$ where $\chi$ is a random colouring. 

\begin{defn}\label{approxdefn}
For a natural number $s$ (which may depend on $n$) we say that a bijection $f$ on $V(Q_n)$ is \em $s$-approximately local \em if for all $v \in V(Q_n)$ there exists a $g(v) \in V(Q_n)$ such that $|f(\Gamma(v)) \cap \Gamma(g(v))| \ge n-s$.  We call the function $g$ a \em dual \em of $f$.

If $f$ has a unique dual $g$, then we write $f_{\star} = g$. Note that this will be the case when $s < \tfrac{n}{2}$. We also define $\mathrm{Local}_s$ as the set of $s$-approximately local functions.
\end{defn}

Note that if $f$ is $s$-approximately local, then the set $\left\{f(w) : w \in \Gamma(v)\right\}$ is clustered around a vertex of $Q_n$, although perhaps not around $f(v)$. Note also that a bijection $f$ being $s$-approximately local where $s$ is small does not force $f$ to be an automorphism.  For example, the map on $Q_{2k}$ that fixes vertices of even weight and maps vertices of odd weight to the antipodal point is $0$-approximately local but not an automorphism.

For the proof of Theorem \ref{f-1-stability}, we will need the following well-known result of Harper \cite{HARP}, which uses the power set $\cP(n)$ interpretation of the hypercube $Q_n$.

\begin{thm}\label{harpermod2}
Let $<_H$ be the ordering of $V(Q_n)$ such that $A <_H B$ if $|A| < |B|$ or if $|A| = |B|$ and $\max ((A \cup B) \setminus (A \cap B)) \in B$. For each $\ell \in \bN$, let $S_{\ell}$ be the first $\ell$ elements of $V(Q_n)$ according to $<_H$. If $D \subset V(Q_n)$ with $|D| = \ell$, then
	\begin{align}
		|\Gamma(D) \cup D| \ge |\Gamma(S_{\ell}) \cup S_{\ell}|. \nonumber
	\end{align}
\end{thm}

An application of this theorem shows that for $A \subset V(Q_n)$ with $|A| \le n$,
	\begin{align}
		|\Gamma(A)\cup A| &\ge 1 + n + {n \choose 2} - {n- (|A|-1) \choose 2} \nonumber \\
		&= 1+ n + \binom{n}{2} - \left(\binom{n-|A|}{2}+n-|A|\right) \nonumber \\
		&= 1 + |A| + \binom{n}{2}-\binom{n-|A|}{2}. \nonumber
	\end{align}
Then, since $|\Gamma(A)| \ge |\Gamma(A) \cup A| - |A|$, we see that
	\begin{align}
		|\Gamma(A)| \ge {n \choose 2} - {n - |A| \choose 2}. \label{harperuse}
	\end{align}
The following result is a simple corollary of Theorem \ref{harpermod2}.

\begin{cor}\label{harper}
Let $r \ge 2$ and let $s(n)$ be a function with $s(n) \rightarrow \infty$ and $s(n) = o(n)$ as $n \rightarrow \infty$. Then there exists a constant $C>0$ such that the following holds: If $A \subseteq V(Q_n)$ is such that $|\Gamma(A)| \le \binom{n}{r} + n^{r-1} s(n)$, then $|A| \le \binom{n}{r-1} + Cn^{r-2}s(n)$.
\end{cor}

\begin{proof}
Suppose that $A \subset V(Q_n)$ with $|A| > {n \choose r-1} + Cn^{r-2}s(n)$ (where $C>0$ is a constant to be specified later). Let $S$ be the first $|A|$ elements of $V(Q_n)$ according to $<_H$. By Theorem \ref{harpermod2},
	\begin{align*}
		|\Gamma(A)| &\ge |\Gamma(A) \cup A| - |A| \\
		&\ge |\Gamma(S) \cup S| - |S|.
	\end{align*}
The set $S$ may be written as $B_{r-1}(0) \cup S'$, where $S'$ is a subset of $[n]^{(r)}$ with $|S'| = |A| - |B_{r-1}(0)|$. Thus we have $\Gamma(S) \cup S = B_r(0) \cup T$, where $T = \Gamma(S') \cap [n]^{(r+1)}$. By the local LYM inequality (see  \cite[Ex. 13.31(b)]{LYMstuff}), $|T| \ge \frac{n-r}{r+1}|S'|$, and so $|\Gamma(S) \cup S| \ge |B_r(0)| + \frac{n-r}{r+1}|S'|$. Therefore for sufficiently large $n$,
	\begin{align*}
		|\Gamma(A)| &\ge |B_r(0)| + \frac{n-r}{r+1}|S'| - \left(|B_{r-1}(0)| + |S'|\right) \\
		&= {n \choose r} + \frac{n-2r-1}{r+1}|S'|.
	\end{align*}
For sufficiently large $n$, $|S'| = |A|-|B_{r-1}(0)| > (C/2)n^{r-2}s(n)$ and $\tfrac{n-2r-1}{r+1} \ge \tfrac{n}{r+2}$, and so
\[
|\Gamma(A)| > {n \choose r} + \frac{C}{2(r+2)}n^{r-1}s(n).
\]
So we see that if $C \ge 2(r+2)$, then $\Gamma(A) > {n \choose r} + n^{r-1}s(n).$
\end{proof}

We are now ready to prove Theorem \ref{f-1-stability}.

\begin{proof}[Proof of Theorem \ref{f-1-stability}]
Let $s(n)$ be a function with $s(n) \rightarrow \infty$ and $s(n) = o(n)$ as $n \rightarrow \infty$, and suppose that $A \subseteq V(Q_n)$ is such that $|A| = n$ and $|\Gamma(A)| \le \binom{n}{2} + ns(n)$. Let $\varepsilon  = \frac{1}{100}$.

Let $Y_1 = \left\{v \in A : |\Gamma(v) \setminus \Gamma(A\setminus v)| \ge (1+\varepsilon) \frac{n}{2}\right\}$. Note that if $D \subset V(Q_n)$ with $|D| \le n$, then \eqref{harperuse} implies that $|\Gamma(D)| \ge |D|\frac{n-1}{2}$. Applying this to $\Gamma(A \setminus Y_1)$ gives
	\begin{align*}
		|\Gamma(A)| &\ge (1+\varepsilon) \frac{n}{2}|Y_1| + |\Gamma(A \setminus Y_1)| \\
		&\ge (1+\varepsilon) \frac{n}{2}|Y_1| + (n-|Y_1|)\frac{n-1}{2} \\
		&\ge \binom{n}{2} + \frac{\varepsilon}{2}|Y_1|n.
	\end{align*}
Since $|\Gamma(A)| \le \binom{n}{2} + ns(n)$, we see that $|Y_1| \le (2/\varepsilon)s(n)$.

Let $A_1 = A \setminus Y_1$ and consider the graph $G=(V,E)$ where $V = A_1$ and $uv \in E$ iff $u$ and $v$ differ in exactly two co-ordinates. We will write $\Gamma_G(v)$ for the neighbourhood of a vertex $V$ in the graph $G$, and reserve $\Gamma$ for the neighbourhood in $Q_n$. In $Q_n$, any two vertices have 2 common neighbours if they are at distance two, and no common neighbours otherwise. Therefore, for all $v\in A_1$,
	\begin{align}
		|\Gamma(v) \cap \Gamma(A\setminus v)| \le 2(|\Gamma_G(v)| + |Y_1|). \label{small}
	\end{align}
Taking $n$ large enough so that $(2/\varepsilon)s(n) \le \varepsilon n$, \eqref{small} gives
	\begin{align}
		|\Gamma_G(v)| &\ge \frac{1}{2}\left(n - \frac{1+ \varepsilon}{2}n - 2|Y_1|\right) \nonumber \\
		& = \frac{n}{4}\left(2-(1+\varepsilon) - 4\varepsilon\right) \nonumber \\
		& = \frac{1-5\varepsilon}{4}n. \nonumber
	\end{align}

Let $Y_2$ be the set of vertices $v$ in $V(G)$ for which there does not exist another vertex $u \in V(G)$ such that $|\Gamma_G(u) \cap \Gamma_G(v)| \ge \varepsilon n$. Suppose that $|Y_2| \ge 5$. Note that if $n$ is sufficiently large, then
	\begin{align*}
		|\Gamma_G(Y_2)| \ge 5\frac{1-5\varepsilon}{4}n - \binom{5}{2}\varepsilon n > |V(G)|.
	\end{align*}
This is a contradiction and so we see that $|Y_2| \le 4$. Letting $A_2 = A_1 \setminus Y_2$ we have that, for large enough $n$, $|A_2| \ge n - (3 / \varepsilon)s(n)$.

Consider a vertex $u \in A_2$ and a vertex $v$ such that $|\Gamma_G(u) \cap \Gamma_G(v)| \ge \varepsilon n$. Taking $n$ large enough so that $\varepsilon n \ge 7$, we see that $|\Gamma_{Q_n}^2(u) \cap \Gamma_{Q_n}^2(v)| \ge 7$ and so $u$ and $v$ are at distance two in the hypercube. Note that if $x$ is also at distance $2$ from both $u$ and $v$, then $u,v,$ and $x$ have a common neighbour. Letting $\Gamma_{Q_n}(u) \cap \Gamma_{Q_n}(v) = \left\{w_1,w_2\right\}$ we see that $\left\{u,v\right\} \cup \left(\Gamma_G(u)\cap\Gamma_G(v)\right) \subseteq \Gamma_{Q_n}(w_1) \cup \Gamma_{Q_n}(w_2)$. Without loss of generality, we may then assume that the neighbourhood of $w_1$ contains $u, v,$ and at least $\varepsilon n/3$ other vertices in $A_2$.

Let $B$ be the set of vertices in $V(Q_n)$ with at least $\varepsilon n/3$ neighbours in $A_2$. Then each vertex in $A_2$ has a neighbour in $B$. Suppose that $B=\left\{w_1,\ldots,w_k\right\}$, then for $\ell \le k$ we have
	\begin{align}
		|\Gamma(\left\{w_1,\ldots,w_\ell\right\})\cap A_2| &\ge \sum_{i \in [\ell]} |\Gamma(w_i)| - \sum_{i\neq j}|\Gamma(w_i)\cap\Gamma(w_j)| \nonumber \\
		&\ge \ell\varepsilon n/3 - {\ell \choose 2} 2 \nonumber \\
		&= \ell\left(\varepsilon n/3 - (\ell-1)\right). \nonumber
	\end{align}
So if $\ell = \lceil 6/ \varepsilon \rceil$, then we have $|\Gamma(\left\{w_1,\ldots,w_\ell\right\})\cap A_2| > n$ for sufficiently large $n$. This is a contradiction since $|A_2| \le n$, and so $k \le 6/ \varepsilon$.

Reorder the $w_i$ so that $w_1$ has the largest neighbourhood in $A_2$. Recursively for $i = 1,\ldots, k$, let $C_i = \Gamma(w_i) \cap A_2 \setminus \bigcup_{j<i} \Gamma(w_j)$. Then since the $C_i$ partition $A_2,$
	\begin{align}
		|\Gamma(A_2)| \ge |\Gamma(C_1)| + |\Gamma(A_2 \setminus C_1)| - \sum_{i =2}^k |\Gamma(C_1) \cap \Gamma(C_i)|. \label{shadow5}
	\end{align}

For $i=2,\ldots,k$, split $C_1$ into $D_i = C_1 \cap \Gamma(w_i)$ and $F_i = C_1 \setminus D_i$. Then $|D_i| \le 2$ for each $i$ and so
	\begin{align}
		\sum_{i =2}^k|\Gamma(D_i)\cap\Gamma(C_i)| \le \sum_{i=2}^k \left|\Gamma(D_i)\right| \le 2kn. \label{shadow4}
	\end{align}

Now consider $|\Gamma(C_i) \cap \Gamma(F_i)| \le 2|\{(u,v) \in C_i \times F_i : u,v$ differ in $2$ co-ordinates$\}|$. If $w_i$ and $w_1$ are at distance at least $5$ from each other, then there can be no $u \in C_i$ and $v \in F_i$ at distance two from each other. The same is true if $w_i$ and $w_1$ are an odd distance from one another (the hypercube is bipartite). Therefore we need only consider the cases when $w_i$ and $w_1$ are distance $2$ or $4$ from each other.

First suppose that $w_i$ and $w_1$ are at distance $2$ from each other and recenter the hypercube so that $w_i = 0$ and $w_1 = e_1+e_2$. If $e_1 \in A_2$, then $e_1 \in C_1 \cap \Gamma(w_i)$ and so $e_1 \not\in F_i$ and $e_1 \notin C_i$. On the other hand, if $e_1 \not\in A_2$, then $e_1$ is not in any $C_j$ and so can be in neither $C_i$ nor $F_i$. The same is true for $e_2$ and so $e_1,e_2 \not\in C_i \cup F_i$. Suppose that $C_i = \left\{e_t: t \in T_i\right\}$ where $|T_i|=|C_i|$ (recall that $w_i=0$). Then for each element $e_t$, the only possible vertex in $F_i$ at distance two from $e_i$ is $e_1+e_2+e_t$. Therefore, $|\{(u,v) \in C_i \times F_i : u,v$ differ in $2$ co-ordinates$\}| \le |C_i|$ and so $|\Gamma(C_i)\cap\Gamma(F_i)| \le 2|C_i|$.

If $w_i$ and $w_1$ are at distance $4$ from each other then $|\Gamma(C_i)\cap\Gamma(F_i)| \le \min\left\{6,3|C_i|\right\}$. In both cases 
	\begin{align}
		|\Gamma(C_i)\cap\Gamma(F_i)| \le 3|C_i|. \label{shadow6}
	\end{align}

Putting \eqref{shadow4} and \eqref{shadow6} into \eqref{shadow5}, and using \eqref{harperuse} gives
	\begin{align}
		|\Gamma\left(A_2\right)| &\ge |\Gamma\left(C_1\right)| + |\Gamma\left(A_2 \setminus C_1\right)| - \sum_{i =2}^k \left(|\Gamma\left(C_i\right) \cap \Gamma\left(D_i\right)| + |\Gamma\left(C_i\right) \cap \Gamma\left(F_i\right)| \right) \nonumber \\
		&\ge {n \choose 2} - {n- |C_1| \choose 2} + {n \choose 2} - {n-\left|A_2 \setminus C_1\right|\choose 2} - 2kn -3n \nonumber \\
		&= \frac{2n^2 - \left(n- |C_1|\right)^2 - \left(n-\left(|A_2 \setminus C_1|\right)\right)^2}{2} + O\left(n\right) \nonumber \\
		&= \frac{2n\left(|C_1|+|A_2\setminus C_1|\right) - |C_1|^2 - |A_2\setminus C_1|^2}{2} +O\left(n\right). \nonumber \\
		&\geq \frac{2n|A_2|- |C_1|^2 - |A_2\setminus C_1|^2}{2} +O\left(n\right). \nonumber
	\end{align}

Recall that $n- \left(3/\varepsilon\right)s\left(n\right) \leq |A_2| \leq n$. Therefore
	\begin{align}
		|\Gamma\left(A_2\right)| &\ge n^2 - \frac{|C_1|^2 + \left(n-|C_1|\right)^2}{2} + O\left(ns\left(n\right)\right) \nonumber \\
		&= {n \choose 2} + |C_1|\left(n-|C_1|\right) + O\left(ns\left(n\right)\right). \nonumber
	\end{align}

Since $|C_1| \ge \varepsilon n/3$, we obtain
	\begin{align}
		|\Gamma\left(A_2\right)| \ge {n \choose 2} + \varepsilon n/3\left(n-|C_1|\right) + O\left(ns\left(n\right)\right). \nonumber
	\end{align}

We started off with the assumption that $|\Gamma(A)| \le \frac{n^2}{2} + ns(n)$ and so we see that $n-|C_1| = O(s(n))$. Finally recall that $C_1 = \Gamma(w_1) \cap A_2 \subseteq \Gamma(w_1)\cap A$ and so we are done.
\end{proof}

An application of Corollary \ref{harper} gives the following corollaries which will later be used in conjunction with Lemma \ref{scott}.

\begin{cor}\label{use-f-1}
Let $s(n)$ be a function with $s(n) \rightarrow \infty$ and $s(n) = o(n)$ as $n \rightarrow \infty$, and let $r \ge 1$. Then there exists a constant $K = K(s(n),r)>0$ such that if $A \subseteq V(Q_n)$ with $|A| = n$ and $|\Gamma^r(A)| \le \binom{n}{r+1} + n^r s(n)$, then there exists some $w \in V(Q_n)$ for which $|\Gamma(w) \cap A| \ge n - Ks(n)$.
\end{cor}

\begin{proof}
We will prove this result by induction on $r$. The base case $r=1$ is just Theorem \ref{f-1-stability} and so we just need to prove the inductive step. Let $s(n)$ be a function with $s(n) \rightarrow \infty$ and $s(n) = o(n)$ as $n \rightarrow \infty$, and let $r > 1$, and suppose $|A| = n$ and $|\Gamma^r(A)| \le \binom{n}{r+1} + n^r s(n)$. Then we may apply Corollary \ref{harper} to $\Gamma^{r-1}(A)$ to see that there is a constant $C$ with $|\Gamma^{r-1}(A)| \le {n \choose r} + Cn^{r-1}s(n)$. The result then follows by the inductive hypothesis.
\end{proof}

\begin{cor}\label{ftoj}
Let $r \ge 1$, and let $s(n)$ be a function with $s(n) \rightarrow \infty$ and $s(n) = o(n)$ as $n \rightarrow \infty$. Then there exists a constant $K>0$ such that any bijection $f: V(Q_n)\rightarrow V(Q_n)$ such that $|\Gamma^r(f(\Gamma(v)))| \le {n \choose r+1} + n^rs(n)$ for all $v \in V(Q_n)$ is $Ks(n)$-approximately local.
\end{cor}

\begin{proof}
Let $r \ge 1$, and let $s(n)$ be a function with $s(n) \rightarrow \infty$ and $s(n) = o(n)$ as $n \rightarrow \infty$. Suppose that $|\Gamma^r(f(\Gamma(v)))| \le \binom{n}{r+1} + n^r s(n)$ for each vertex $v \in V(Q_n)$. By Corollary \ref{use-f-1}, there exists a constant $K>0$ such that for all $v\in V(Q_n)$, there exists a $g(v) \in V(Q_n)$ such that $|\Gamma(g(v)) \cap f(\Gamma(v))| \ge n - Ks(n)$. Then $g$ is the dual of $f$ realising that $f \in \mathrm{Local}_{Ks(n)}$.
\end{proof}

While Corollary \ref{ftoj} is needed for our proof of Theorem \ref{main}, it is not enough for Theorem \ref{1ball} where we will need to allow $s(n) = \Theta(n)$. It would be helpful to have a result similar to Theorem \ref{f-1-stability} in this case. Here, we prove a result with the added condition that the set $A$ does not cluster too much around two different vertices.

\begin{lem}\label{1-stab}
Let $t(n)\ge 5$ be a function with $t(n) \rightarrow \infty$ and $t(n) = o(n)$ as $n \rightarrow \infty$, and let $s(n)$ be a function on the natural numbers such that $1-2s(n)n^{-1} - 14\sqrt{t(n)/n} \ge 0$. Suppose that $A \subseteq V(Q_n)$ with $|A|=n$ and $|\Gamma(A)| \le \binom{n}{2} + s(n)n$, and suppose there do not exist distinct $w_1,w_2 \in V(Q_n)$ such that $|A\cap \Gamma(w_i)| > t(n)$ for $i=1,2$. Then there exists some $w \in V(Q_n)$ for which
	\begin{align*}
		|\Gamma(w) \cap A| \ge n\left(1-2s(n)n^{-1} - 14\sqrt{t(n)/n}\right)^{\frac{1}{2}}.
	\end{align*}
\end{lem}

\begin{proof}
Let $G = (A,E)$ where $uv \in E$ if and only if $d(u,v) = 2$. Suppose that $A_1$ is a largest clique in $G$. A clique of size at least 5 in $G$ corresponds to a collection of vertices in $A$ in the $Q_n$-neighbourhood of a single vertex. Then by assumption all cliques other than $A_1$ have size at most $t(n)$. Let $A'=\left\{v \in A \setminus A_1 : \deg_G(v) \ge 3\sqrt{nt(n)}\right\}$.

Suppose there exist distinct $u,v \in A'$ with $|\Gamma_G(u) \cap \Gamma_G(v)| \ge 2t(n)$. Then $t(n) \ge 5$, and so $|\Gamma_G(u) \cap \Gamma_G(v)| \ge 10$, which corresponds to there being at least $10$ vertices at distance $2$ from both $u$ and $v$ in $Q_n$. This is only possible if $u$ and $v$ are at distance $2$ in $Q_n$. Without loss of generality assume that $u = \emptyset$ and $v = \{1,2\}$. Then every vertex $x \in \Gamma_G(u) \cap \Gamma_G(v)$ contains two elements, exactly one of which is $1$ or $2$. So then $x$ is a neighbour of either $1$ or $2$ in $Q_n$, and by the pigeonhole principle one of $1$ and $2$ (without loss of generality assume $1$) has at least $t(n)$ $Q_n$-neighbours in $A$. But then these $Q_n$-neighbours of $1$ plus $u$ and $v$ form a clique in $G$ of size at least $t(n)+2$. This cannot be since there is no clique of size $t(n)+2$ not entirely contained in $A_1$.

Therefore $|\Gamma_G(u) \cap \Gamma_G(v)| < 2t(n)$ for each $u,v \in A'$. But now, for any $Y \subseteq A'$, we have
	\begin{align*}
		|\Gamma_G(Y)| &\ge \sum_{v \in Y} \deg_G(v) - \sum_{v\neq w \in Y} |\Gamma_G(v) \cap \Gamma_G(w)| \\
		&\ge 3\sqrt{nt(n)}|Y| - t(n)|Y|^2.
	\end{align*}
So we see that if $|Y| = \left\lceil \sqrt{n/t(n)} \right\rceil$, then we have $|\Gamma_G(Y)| > n$. This gives a contradiction, so we must have $|A'| \le \sqrt{n/t(n)}$.

Note that if $v \in A \setminus (A_1 \cup A')$, then $|\Gamma(v)\setminus \Gamma(A \setminus \left\{v\right\})| \ge n - 2\deg_G(v) \ge n- 6\sqrt{nt(n)}$. We can now give a lower bound for $|\Gamma(A)|$ in terms of $t(n)$ and $|A_1|$ by applying \eqref{harperuse}. Indeed,
	\begin{align*}
		|\Gamma(A)| &\ge |\Gamma(A_1)| + \sum_{v \in A \setminus (A_1 \cup A')} |\Gamma(v)\setminus \Gamma(A \setminus \left\{v\right\})| \\
		&\ge \binom{n}{2} - \binom{n-|A_1|}{2} + \left(n-|A_1|-|A'|\right)\left(n-6\sqrt{nt(n)}\right) \\
		&\ge \binom{n}{2} - \frac{\left(n-\left|A_1\right|\right)^2}{2} + \left(n-|A_1|-\sqrt{n/t(n)}\right)\left(n-6\sqrt{nt(n)}\right) \\
		&\ge \binom{n}{2} - \frac{n^2 - 2n|A_1| + |A_1|^2}{2} + n^2 - n|A_1| -7n^{\frac{3}{2}}t(n)^{\frac{1}{2}} \\
		&= \binom{n}{2} + \frac{n^2 - |A_1|^2}{2} -7n^{\frac{3}{2}}t(n)^{\frac{1}{2}}.
	\end{align*}
Recall that $|\Gamma(A)| \le \binom{n}{2} + s(n)n$, and so
	\begin{align*}
		\frac{n^2 - |A_1|^2}{2} - 7n^{\frac{3}{2}}t(n)^{\frac{1}{2}} \le s(n)n.
	\end{align*}
Rearranging this gives
	\begin{align*}
		|A_1|^2 \ge n^2\left(1-2s(n)n^{-1} - 14\left(\frac{t(n)}{n}\right)^{\frac{1}{2}}\right),
	\end{align*}
and we are done by taking square roots.
\end{proof}

\begin{defn}\label{monodef}
Define $\mathrm{Mono}_s^t$ (where $s$ and $t$ may depend on $n$) as the set of bijections $f \in \mathrm{Cluster}^1_s$ for which, for all $v \in V(Q_n)$, there exists at most one vertex $w \in V(Q_n)$ such that $|f(\Gamma(v))\cap \Gamma(w)| > t$.
\end{defn}

We then have the following direct corollary of Lemma \ref{1-stab}.

\begin{cor}\label{1ftoj}
Let $t(n)\ge 5$ be a function with $t(n) \rightarrow \infty$ and $t(n) = o(n)$ as $n \rightarrow \infty$, and let $s(n)$ be a function on the natural numbers such that $1-2s(n)n^{-1} - 14\sqrt{t(n)/n} \ge 0$. Then $\mathrm{Mono}_{s(n)n}^{t(n)} \subseteq \mathrm{Local}_{\alpha(n)n}$ where
	\begin{align*}
		\alpha(n) = 1-\left(1-2s(n)n^{-1} - 14\sqrt{t(n)/n}\right)^{\frac{1}{2}}.
	\end{align*}
Further, if $\alpha(n)n < n - t(n)$, then a function $f \in \mathrm{Mono}_{s(n)n}^{t(n)}$ has at most one dual.
\end{cor}

The following lemma shows that the inverse of an approximately local bijection is itself approximately local.

\begin{lem}\label{a.l.inverse}
Let $s$ be some natural number. If $f \in \mathrm{Local}_s$ has a bijective dual $g$, then $f^{-1} \in \mathrm{Local}_s$ and $g^{-1}$ is a dual of $f^{-1}$.
\end{lem}

\begin{proof}
Note that for all $w \in V(Q_n)$, $|f(\Gamma(w)) \cap \Gamma(g(w))| \ge n- s$ and so, since $f$ is a bijection, $|\Gamma(w) \cap f^{-1}(\Gamma(g(w)))| \ge n-s$. Now let $v \in V(Q_n)$ and suppose that $v = g(u)$. Then $f^{-1}(\Gamma(v)) = f^{-1}(\Gamma(g(u)))$, and so
	\begin{align*}
		|f^{-1}\left(\Gamma\left(v\right)\right) \cap \Gamma\left(g^{-1}\left(v\right)\right)| = |f^{-1}\left(\Gamma\left(g\left(u\right)\right)\right) \cap \Gamma\left(u\right)| \ge n-s.
	\end{align*}
Since $v$ was an arbitrary vertex of the hypercube, we can conclude that $f^{-1}$ is $s$-approximately local and has $g^{-1}$ as one of its duals.
\end{proof}

We now use Theorem \ref{f-1-stability} to show that $s(n)$-approximately local bijections have $O(s(n))$-approximately local duals.

\begin{lem}\label{g-stability}
Let $s(n)<n/2$ be a function with $s(n) \rightarrow \infty$ and $s(n) = o(n)$ as $n \rightarrow \infty$. Then there exists some constant $K>0$ such that for every $s(n)$-approximately local bijection $f$, the dual $f_{\star}$ is $Ks(n)$-approximately local.
\end{lem}

\begin{proof}
We will show that $f_{\star}^{-1}$ is $Ks(n)$-approximately local, and then apply Lemma \ref{a.l.inverse}. Let $s(n)<n/2$ be a function with $s(n) \rightarrow \infty$ and $s(n) = o(n)$ as $n \rightarrow \infty$. Suppose that $f \in \mathrm{Local}_{s(n)}$ and let $g = f_{\star}$ (so that for all $v \in V(Q_n)$, $|f(\Gamma(v))\cap \Gamma(g(v))| \ge n-s(n)$). Fix some $v \in V(Q_n)$. For each $w \in \Gamma(v)$, writing $w' = g^{-1}(w)$, we have $|\Gamma(w) \cap f(\Gamma(w'))| \ge n-s(n)$ and so $|\Gamma^2(v) \cap f(\Gamma(w'))| \ge n-s(n)$. Let $R_w = f(\Gamma(w'))\setminus \Gamma^2(v)$, so $|R_w| \le s(n)$. Now
	\begin{align*}
		f\left(\Gamma\left(g^{-1}\left(\Gamma\left(v\right)\right)\right)\right) & = \bigcup_{w \in \Gamma (v)} f\left(\Gamma \left(g^{-1}(w)\right)\right)  \\
		& \subseteq \Gamma^2\left(v\right) \cup \bigcup_{w \in \Gamma\left(v\right)}R_w.
	\end{align*}
Since $f$ is a bijection, we see that
	\begin{align*}
		|\Gamma\left(g^{-1}\left(\Gamma\left(v\right)\right)\right)| \le \binom{n}{2} + ns\left(n\right).
	\end{align*}
Since $g^{-1}(\Gamma(v)) \subseteq V(Q_n)$ is a subset of size $n$, we may appeal to Theorem \ref{f-1-stability} to see that there exists some $w \in V(Q_n)$ such that $|\Gamma(w)\cap g^{-1}(\Gamma(v))| = n - O(s(n))$. Then $g^{-1}$ is $O(s(n))$-approximately local. Since $s(n) = o(n)$, it follows that $g^{-1}$ must have a unique, bijective dual. By Lemma \ref{a.l.inverse}, we conclude that $g$ is $O(s(n))$-approximately local.
\end{proof}

\begin{defn}\label{diagdefn}
For an $s(n)$-approximately local bijection $f$, we say that $f$ is \em diagonal \em if it is the dual of its dual, i.e. if $f_{\star \star} = f$.
\end{defn}

For a natural number $s$ (which may depend on $n$), let $\mathrm{Diag}_s$ be the set of diagonal bijections in $\mathrm{Local}_s$. The next two results will show that an $s(n)$-approximately local diagonal bijection induces large rigid structures within the hypercube.

\begin{cor}\label{bothways}
Let $s(n)$ be a function with $s(n) \rightarrow \infty$ and $s(n) = o(n)$ as $n \rightarrow \infty$. Then there exists a constant $K>1$ such that the following holds: Suppose $f$ is an $s(n)$-approximately local diagonal bijection and let $G = (V(Q_n),E')$ where
	\begin{align*}
		E' = \left\{uv \in E(Q_n) : f(u)f_{\star}(v),f(v)f_{\star}(u) \in E(Q_n)\right\}.
	\end{align*}
Then $G$ has minimum degree at least $n-Ks(n)$.
\end{cor}

\begin{proof}
Let $f$ be an $s(n)$-approximately local diagonal bijection. By Lemma \ref{g-stability}, there exists some $K' >0$ such that $f_{\star}$ is $K's(n)$-approximately local. Now pick $v \in V(Q_n)$ and note that
	\begin{align}
		\deg_G\left(v\right) \ge n - \left(|\Gamma\left(f_{\star}\left(v\right)\right) \setminus f\left(\Gamma\left(v\right)\right)| + |\Gamma\left(f\left(v\right)\right) \setminus f_{\star}\left(\Gamma\left(v\right)\right)|\right). \label{bothways1}
	\end{align}
Since $f \in \mathrm{Local}_{s(n)}$ and $f_{\star} \in \mathrm{Local}_{K's(n)}$, $|\Gamma(f_{\star}(v)) \setminus f(\Gamma(v))| \le s(n)$ and $|\Gamma(f_{\star \star}(v)) \setminus f_{\star}(\Gamma(v))| \le K's(n)$. Recall that $f$ is diagonal, so $f_{\star \star} = f$ and $\Gamma(f_{\star \star}(v)) \setminus f_{\star}(\Gamma(v)) = \Gamma(f(v)) \setminus f_{\star}(\Gamma(v))$. Putting these inequalities into \eqref{bothways1}, we see that $\deg_G(v) \ge n - Ks(n)$, where $K = K'+1$.
\end{proof}

\begin{lem}\label{krigid}
Let $s(n)$ be a function with $s(n) \rightarrow \infty$ and $s(n) = o(n)$ as $n \rightarrow \infty$, and suppose $G = (V(Q_n),E')$ is a subgraph of the hypercube with minimum degree at least $n -s(n)$. For a vertex $v \in V(Q_n)$, let $R_0(v) = \left\{v\right\}$, and then recursively for $i \ge 1$ let 
	\begin{align}
		R_i(v) = \left\{w \in \Gamma_{Q_n}^i(v) : \Gamma_{Q_n}(w) \cap \Gamma_{Q_n}^{i-1}(v) = \Gamma_{G}(w) \cap R_{i-1}(v)\right\}. \label{layering}
	\end{align}
Then $|R_k(v)| \ge \binom{n}{k} - en^{k-1}s(n)$ for all $k\ge 1$.
\end{lem}

Note that $w \in R_i(v)$ if and only if $w$ is at distance $i$ from $v$ in the hypercube, and $G$ contains all shortest $vw$ paths found in the hypercube.

\begin{proof}
We will show by induction on $k$ that $|R_k(v)| \ge \binom{n}{k} - Y_kn^{k-1}s(n)$ where $Y_1 = 1$ and inductively for $i > 1$, $Y_{i+1} = \frac{1}{i!} + Y_i = \sum_{j=1}^i\frac{1}{j!}$ (so then $Y_k \le e$ for all $k$). The base case $k=1$ follows directly from the minimum degree condition, giving $Y_1 = 1$.

So suppose the result holds for $k \le m$ (so that $|R_k(v)| \ge \binom{n}{k} - Y_k n^{k-1}s(n)$ for all $v \in V(Q_n)$ and $k \le m$) and consider $x \in \Gamma_{Q_n}^{m+1}(v) \setminus R_{m+1}(v)$. Then either there is an edge missing between $\Gamma_{Q_n}^m(v)$ and $x$ in $G$, or there is a vertex $w \in \Gamma_{Q_n}^m(v) \setminus R_m(v)$ with $x \in \Gamma_{Q_n}(w)$. We therefore have the following relation
	\begin{align*}
		\Gamma_{Q_n}^{m+1}(v) \setminus R_{m+1}(v) &\subset \bigcup_{u \in \Gamma_{Q_n}^m(v)}(\Gamma_{Q_n}(u)\setminus \Gamma_G(u)) \cup \bigcup_{w \in \Gamma_{Q_n}^m(v) \setminus R_m(v)}\Gamma_{Q_n}(w).
	\end{align*}
The inductive hypothesis then gives
	\begin{align*}
		|\Gamma_{Q_n}^{m+1}(v) \setminus R_{m+1}(v)| &\le \binom{n}{m}s(n) + Y_mn^{m-1}s(n)n \\
		&\le \left(\frac{1}{m!}+Y_m\right)n^ms(n) = Y_{m+1}n^ms(n). 
	\end{align*}
Thus $|R_{m+1}(v)| \ge \binom{n}{m+1} - Y_{m+1}n^ms(n) \ge \binom{n}{m+1} - en^ms(n)$.
\end{proof}

Suppose that there is a colouring $\chi$ and an $s(n)$-approximately local bijection $f$ such that $f \in \mathrm{Isom}^{(2)}(\chi)$. The next lemma shows that the $\chi$-colouring of a $2$--ball around a vertex $v \in V(Q_n)$ differs by $O(ns(n))$ from the $\chi$-colouring of the $2$--ball around $f_{\star \star}^{-1}(f(v))$. Note that there is no ambiguity in writing $f_{\star \star}^{-1}$, as Lemma \ref{a.l.inverse} tells us that $(g_{\star})^{-1} = (g^{-1})_{\star}$. This result will later allow us to consider only diagonal bijections and so be able to apply Lemma \ref{krigid}.

\begin{lem}\label{f-h-inverse}
Let $s(n)$ be a function with $s(n) \rightarrow \infty$ and $s(n) = o(n)$ as $n \rightarrow \infty$, and let $f \in \mathrm{Local}_{s(n)}$. If $\chi : V(Q_n) \rightarrow \left\{0,1\right\}$ is such that $f \in \mathrm{Isom}^{(2)}(\chi)$, then for all $v \in V(Q_n)$, $$d(\chi^{(2)}(v),\chi ^{(2)}(f_{\star \star}^{-1}(f(v)))) = O(ns(n)).$$
\end{lem}

\begin{proof}
Let $f$ be an $s(n)$-approximately local bijection and let $\beta = f^{-1}$. Let $g = \beta_{\star}$. By Lemmas \ref{a.l.inverse} and \ref{g-stability}, there is a $K>0$ such that $g$ is $Ks(n)$-approximately local. Let $h = g_{\star}=\beta_{\star \star}$ be the dual of $g$.

Let $v \in V(Q_n)$, $w = f(v)$, and let $S = \left\{i : g(w+e_i) \in \Gamma(h(w))\right\}$. Note that $|S| \ge n - Ks(n)$ since $g$ is $Ks(n)$-approximately local. Then let $\pi^{\star}$ be a permutation on $[n]$ such that $g(w+e_i) = h(w)+e_{\pi^{\star}(i)}$ for all $i \in S$.

For each $i \in S$, let $T^i = \left\{j : \beta(w+e_i+e_j) \in \Gamma(g(w+e_i))\right\}$. Note that $|T^i| \ge n -s(n)$ for each $i$ since $\beta$ is $s(n)$-approximately local. Then let $\pi^i$ be a permutation on $[n]$ such that $\beta(w+e_i+e_j) = g(w+e_i)+e_{\pi^i(j)}$ for all $j \in T^i$.

If $i \in S$ and $j \in T^i$, then
	\begin{align}
		\beta\left(w+e_i+e_j\right) &= g\left(w+e_i\right)+e_{\pi^i\left(j\right)} \nonumber \\
		&= h\left(w\right) + e_{\pi^{\star}\left(i\right)} + e_{\pi^i\left(j\right)}. \nonumber
	\end{align}
Analogously, if $j \in S$ and $i \in T^j$, then $\beta(w+e_i+e_j) = h(w) + e_{\pi^{\star}(j)} + e_{\pi^j(i)}$. We then have $e_{\pi^{\star}(j)} + e_{\pi^j(i)} = e_{\pi^{\star}(i)} + e_{\pi^j(j)}$. Since $e_{\pi^{\star}(i)} \neq e_{\pi^{\star}(j)}$, we must have $e_{\pi^{\star}(i)} = e_{\pi^j(i)}$ and $e_{\pi^{\star}(j)} = e_{\pi^i(j)}$. Therefore $\beta(w+e_i+e_j) = h(w) + e_{\pi^{\star}(i)} + e_{\pi^{\star}(j)}$. Now, let 
	\begin{align}
		W = \left\{w+e_i+e_j : i \neq j \in [n], \beta\left(w+e_i+e_j\right) = h\left(w\right)+e_{\pi^{\star}\left(i\right)}+e_{\pi^{\star}\left(j\right)} \right\}. \nonumber
	\end{align}
If $w + e_i + e_j \not\in W$, then it must be that either $i$ and $j$ are not both in $S$, or $i$ is not in $T^j$, or $j$ is not in $T^i$. Hence we can bound $\Gamma^2(w)\setminus W$ as follows.
	\begin{align}
		\Gamma^2(w)\setminus W &\subseteq \left\{w + e_i+e_j : \left\{i,j\right\} \not\subseteq S\right\} \cup \left\{w + e_i+e_j : i\in S, j \not\in T^i\right\} \nonumber \\
		&= \left\{w + e_i+e_j : \left\{i,j\right\} \not\subseteq S\right\} \cup \bigcup_{i\in S} \left\{w + e_i+e_j : j \not\in T^i\right\}. \label{herror1}
	\end{align}
Recall that $|S| \ge n-Ks(n)$ and so since $s(n) = o(n)$
	\begin{align}
		| \left\{w + e_i+e_j : \left\{i,j\right\} \not\subseteq S\right\}| = {n \choose 2} - {|S| \choose 2} \le Kns(n)(1+o(1)). \label{herror2} 
	\end{align}
Similarly $|T^i| \ge n-s(n)$ for all $i \in S$ and so
	\begin{align}
		\left|\bigcup_{i\in S} \left\{w + e_i+e_j : j \not\in T^i\right\}\right| \le ns(n). \label{herror3}
	\end{align}
Combining \eqref{herror1}, \eqref{herror2} and \eqref{herror3} we see that
	\begin{align}
	\label{herrorTotal}
		|\Gamma^2(w)\setminus W| \le (1+K)ns(n)(1+o(1)).
	\end{align}

Now suppose also that $f \in \mathrm{Isom}^{(2)}(\chi)$. Then $\chi^{(2)}(v) \cong \chi_f^{(2)}(w)$ and so there exists an isomorphism $y$ from $B_2(v)$ to $B_2(w)$ such that $(\chi_f \circ y)\restriction_{B_2(v)}=\chi \restriction_{B_2(v)}$. Let $\rho$ be a permutation on $[n]$ such that $y(v + e_j) = w+e_{\rho(j)}$ for each $j\in[n]$. Then for distinct $i,j \in [n]$
	\begin{align}
		\chi(v + e_i + e_j) = \chi_f(w+e_{\rho(i)} + e_{\rho(j)}). \label{qwerty}
	\end{align}

Let $W^{\rho} = \left\{v+e_{\rho^{-1}(a)}+e_{\rho^{-1}(b)} : w + e_a + e_b \in W\right\}$, so that clearly $|W^{\rho}| = |W|$. Recall that for $w + e_i + e_j \in W$ we have $$w + e_i + e_j = f(h(w) + e_{\pi^{\star}(i)}+e_{\pi^{\star}(j)}).$$ Combining this with \eqref{qwerty} gives, for $v+e_{\rho^{-1}(i)}+e_{\rho^{-1}(j)} \in W^{\rho}$
	\begin{align*}
		\chi\left(v+e_{\rho^{-1}\left(i\right)}+e_{\rho^{-1}\left(j\right)}\right) &= \chi_f\left(w + e_i+e_j\right) \\
		&= \chi_f\left(f\left(h\left(w\right)+e_{\pi^{\star}\left(i\right)}+e_{\pi^{\star}\left(j\right)}\right)\right) \\
		&= \chi\left(h\left(w\right) + e_{\pi^{\star}\left(i\right)}+e_{\pi^{\star}\left(j\right)}\right).
	\end{align*}
	
Now $\zeta(v+e_{\rho^{-1}(i)}+e_{\rho^{-1}(j)}) = h(w)+e_{\pi^{\star}(i)}+e_{\pi^{\star}(j)}$ defines an isomorphism between $B_2(v)$ and $B_2(h(w))$. Further, we have
\[
\chi\left(v+e_{\rho^{-1}\left(i\right)}+e_{\rho^{-1}\left(j\right)}\right) = \chi \circ \zeta \left(v+e_{\rho^{-1}\left(i\right)}+e_{\rho^{-1}\left(j\right)}\right),
\]
for each $v +e_{\rho^{-1}(i)}+e_{\rho^{-1}(j)} \in W^{\rho}$. Therefore $D(\chi \restriction_{B_2(v)},(\chi \circ \zeta)\restriction_{B_2(v)}) \le (\binom{n}{2} - |W^{\rho}|)+n+1$, and so $d(\chi^{(2)}(v),\chi^{(2)}(h(w))) \le |\Gamma^2(w)\setminus W|+n+1$. It follows from \eqref{herrorTotal} and the definition of $h$ that
\[
d(\chi^{(2)}(v),\chi ^{(2)}(f_{\star \star}^{-1}(f(v)))) \le (1+K)ns(n)(1+o(1)) = O(ns(n)).
\]
\end{proof}

\section{Proof of Theorem \ref{main}}\label{combination}
In this section we prove Theorem \ref{main} by combining the probabilistic and structural results proved in Sections \ref{probability} and \ref{structure} respectively. Much of the work has already been done for this. Indeed, by Lemma \ref{scott} and Corollary \ref{ftoj} we may assume that if $f \in \mathrm{Isom}^{(2)}(\chi)$, then $f$ is $s(n)$-approximately local, for some $s(n) = o(n)$. 

For a graph $G=(V,E)$ we say that a subset of the vertices $A \subseteq V$ is \em$t$--spread \em if $A \cap B_{t-1}(u) = u$ for all $u \in A$ (so then all pairs of vertices in $A$ cannot be joined by a path of length $t-1$ or less). We start with a simple proposition which allows us to cover a fraction of the 10th neighbourhood of a vertex with $6$-spread large sets.

\begin{prop}\label{greedycovering}
Let $\delta,\varepsilon > 0$ be such that $2\varepsilon \delta < \frac{1}{10!}$. Then for sufficiently large $n$, there exists a collection of disjoint sets $(A_i)_{i \in J}$ where $J = \{1,\ldots,,\lceil \varepsilon n^6 \rceil\}$, such that each $A_i$ is a $6$-spread subset of $[n]^{(10)}$ and $|A_i| = \lceil \delta n^4 \rceil.$
\end{prop}

A greedy algorithm easily proves this result, but a nicer proof is an application of a result of Hajnal and Szemer\'{e}di.

\begin{thm}[Hajnal-Szemer\'{e}di \cite{hajsze}]\label{hajszethm}
Let $G = (V,E)$ be a graph on $n$ vertices with maximum degree $\Delta$. Then for any $k > \Delta,$ there exists a proper $k$-colouring of $G$ with colour classes all of size $\left\lceil \frac{n}{k} \right\rceil$ or $\left\lfloor \frac{n}{k} \right\rfloor$.
\end{thm}

\begin{proof}[Proof of Proposition \ref{greedycovering}]
Define the graph $G$ on the vertex set $[n]^{(10)}$, where two vertices are connected if they are at Hamming distance at most five from one another. The neighbourhood of a vertex $v$ is contained within the $5$-ball around $v,$ and so the maximum degree in $G$ is at most $n^5$. Let $k = \lceil \varepsilon n^6 \rceil$, and take $n$ large enough so that $k > n^5$. By Theorem \ref{hajszethm}, there exists a $k$-colouring with colour classes $C_1,\ldots,C_k$ of size $\left\lceil{n\choose 10}k^{-1}\right\rceil$ or $\left\lfloor{n\choose 10}k^{-1}\right\rfloor$. Each colour class $C_i$ is a $6$-spread subset of $[n]^{(10)}$ and has size at least $\left\lfloor{n\choose 10}k^{-1}\right\rfloor$. For $n$ sufficiently large $\left\lfloor{n\choose 10}k^{-1}\right\rfloor \ge \frac{n^4}{2\varepsilon 10!} > \delta n^4$. Therefore, for each $i=1,\ldots,k,$ we can take a $6$-spread subset $A_i \subseteq C_i$ of size $|A_i| = \lceil \delta n^4\rceil$.
\end{proof}

Recall that a colouring $\chi$ of the hypercube is $2$-indistinguishable if there is a bijection $f$ for which $\chi_f$ and $\chi$ are $2$-locally equivalent and there exist two non-adjacent vertices $u,v$ such that $f(u)$ and $f(v)$ are adjacent in the hypercube.

\begin{proof}[Proof of Theorem \ref{main}]
Let $\varepsilon > 0$ and let $p = p(n)$ satisfy $n^{-1/4 + \varepsilon} \le p(n) \le 1/2$ for sufficiently large $n.$ Let $\chi$ be a random $(p,1-p)$-colouring of the hypercube $Q_n$. By Lemma \ref{scott}, there is a $K>0$ such that with high probability, for every $f \in \mathrm{Isom}^{(2)}(\chi)$ we have $f^{-1} \in \mathrm{Cluster}^2_{Kn^2p^{-1}\log n}$. Let $s(n) = \frac{\log n}{p}$ (so $s \rightarrow \infty$ and $s = o(n)$ as $n \rightarrow \infty$). We have
	\begin{align*}
		\bP\left[\chi \ \mathrm{is} \mbox{ $2$-indist.}\right] = \bP\left[\exists f \in \mathrm{Isom}^{(2)}(\chi) \ \mathrm{s.t.} \ f^{-1}\in \mathrm{Cluster}^2_{Kn^2s }, \chi \circ f^{-1} \not\cong \chi\right] + o(1).
	\end{align*}
By Corollary \ref{ftoj}, there exists a $K'>0$ such that $\mathrm{Cluster}^2_{Kn^2s } \subseteq \mathrm{Local}_{K's }$, so that
	\begin{align*}
		\bP\left[\chi \ \mathrm{is} \mbox{ $2$-indist.}\right] = \bP\left[\exists f \in \mathrm{Isom}^{(2)}(\chi) \ \mathrm{s.t.} \ f^{-1} \in \mathrm{Local}_{K's },\chi \circ f^{-1} \not\cong \chi\right] + o(1).
	\end{align*}
Then by Lemma \ref{a.l.inverse} we can express this entirely in terms of $f$:
	\begin{align*}
		\bP\left[\chi \ \mathrm{is} \mbox{ $2$-indist.}\right] = \bP\left[\exists f\in \mathrm{Isom}^{(2)}(\chi) \ \mathrm{s.t.} \ f \in \mathrm{Local}_{K's },\chi \circ f^{-1} \not\cong \chi\right] + o(1).
	\end{align*}
Suppose that there exists such an $f \in \mathrm{Local}_{K's } \setminus \mathrm{Diag}_{K's }$, and pick a vertex $v\in V(Q_n)$ such that $f_{\star \star}^{-1}\circ f (v) \neq v$. If $f \in \mathrm{Isom}^{(2)}(\chi)$, then by Lemma \ref{f-h-inverse}, $d(\chi^{(2)}(v),\chi ^{(2)}(f_{\star \star}^{-1}(f(v)))) = O(ns(n))$. But by Lemma \ref{strongunique}, the probability that there is a pair of distinct vertices $x,y$ with $d(\chi^{(2)}(x),\chi^{(2)}(y)) < \frac{n^2p(1-p)}{2}$ is $o(1)$. Since $s(n) = \tfrac{\log n}{p}$ and $p \ge n^{-1/4}$ for sufficiently large $n$, we get that the probability we can choose $f \in \mathrm{Isom}^{(2)}(\chi)$ with  $f \in \mathrm{Local}_{K's } \setminus \mathrm{Diag}_{K's }$ and $\chi \circ f^{-1} \not\cong \chi$ is $o(1)$.

Thus
	\begin{align*}
		\bP\left[\chi \ \mathrm{is} \ \mbox{ $2$-indist.}\right] = \bP\left[\exists f\in \mathrm{Isom}^{(2)}(\chi) \ \mathrm{s.t.} \ f \in \mathrm{Diag}_{K's },\chi \circ f^{-1} \not\cong \chi\right] + o(1).
	\end{align*}

Suppose that $f \in \mathrm{Isom}^{(2)}(\chi)$ with $f \in \mathrm{Diag}_{K's }$, and let $g = f_{\star}$. Recall that by Lemma \ref{g-stability} there exists a constant $L>0$ such that $g \in \mathrm{Local}_{Ls }$. As in Corollary \ref{bothways}, we let $G = (V(Q_n),E')$ where $$E' = \left\{xy \in E\left(Q_n\right) : f(x)g(y),f(y)g(x) \in E\left(Q_n\right)\right\}.$$ Then $G$ has minimum degree at least $n-Ms$ for some constant $M$. Furthermore, define $R_k(w)$ as Lemma \ref{krigid} (see \eqref{layering}). So $|R_k(w)| \ge {n \choose k} - eMn^{k-1}s$.

For each $u \in V(Q_n)$, let $\pi_u$ be a permutation on $[n]$ such that $g(u + e_j) = f(u)+e_{\pi_u(j)}$ for all $j$ such that $u+e_j \in R_1(u)$. We claim that for $k > 1$ odd, for each $w \in R_k(u),$ the vertex $g(w)$ is uniquely determined by the sequence $(f(w))_{w \in R_{k-1}(u)}$. Indeed, suppose that $w = u + \sum_{j=1}^k e_{i_j}$ is in $R_k(u)$. Then $\Gamma_{Q_n}(w) \cap \Gamma^{k-1}_{Q_n}(u) = \Gamma_G(w) \cap R_{k-1}(u)$. Then for all $\ell \in [k]$, $u + \sum_{j \in [k] \setminus \ell} e_{i_j} \in R_{k-1}(u)$ and $g(w)f(u + \sum_{j \in [k] \setminus \ell} e_{i_j}) \in E(Q_n)$. However, there is a unique vertex in the hypercube adjacent to $f(u + \sum_{j \in [k] \setminus \ell} e_{i_j})$ for all $\ell$, and so $g(w)$ is determined by $(f(w))_{w \in R_{k-1}(u)}$. We may similarly say that when $k>1$ is even, $(f(w))_{w \in R_{k}(u)}$ can be determined by $(g(w))_{w\in R_{k-1}(u)}$ (note that when $k=2$, there may be a choice of two vertices adjacent to both $g(u + e_i)$ and $g(u + e_j)$, but one of these is $f(u)$).

Inductively for $k \ge 0$ we then have 
	\begin{align}
		f\left(u + \sum_{j\in S}e_j\right) = f(u)+\sum_{j\in S}e_{\pi_u(j)}, \label{star}
	\end{align}
for all $S \in [n]^{(2k)}$ such that $u+\sum_{j\in S}e_j \in R_{2k}(u)$, and
	\begin{align}
		g\left(u + \sum_{j\in T}e_j\right) = f(u)+\sum_{j\in S}e_{\pi_u(j)}, \nonumber
	\end{align}
for all $T \in [n]^{(2k+1)}$ such that $u+\sum_{j\in T}e_j \in R_{2k+1}(u)$. (For example, if $u + e_1+e_2 + e_3 \in R_3(u)$, then $g(u + e_1+ e_2 + e_3)$ is adjacent to $f(u+e_1+e_2)$, $f(u+e_1+e_3)$ and $f(u+e_2 + e_3)$. By the inductive hypothesis, $f(u+e_1+e_2)=f(u) +e_{\pi_u(1)} + e_{\pi_u(2)}$, $f(u+e_1+e_3)=f(u) +e_{\pi_u(1)} + e_{\pi_u(3)}$, and $f(u+e_2+e_3)=f(u) +e_{\pi_u(2)} + e_{\pi_u(3)}$. There is only one vertex adjacent to all three, and so $g(u+e_1+e_2+e_3) = f(u) + e_{\pi_u(1)} + e_{\pi_u(2)} + e_{\pi_u(3)}$.)

Fix two non-adjacent vertices $u,v \in V(Q_n)$. Our goal is to show that $f(u)$ and $f(v)$ cannot be adjacent. We do this by first showing that if $f(u)$ and $f(v)$ are adjacent, then there are rigid structures around each which are adjacent. We then take substructures of these rigid structures which are $6$-spread (this will allow us to say that the colouring of the $2$-balls around the vertices of these substructures are independent from one another). Finally we consider that if two vertices are adjacent, the colour of one has to fit in with the colouring of the $2$-ball around the other. We are then able to show that this cannot happen with high probability (helped greatly by the independence attained by restricting ourselves to the specified substructures).

Let $C = \{S \in [n]^{(10)} : u+\sum_{j\in S}e_{{\pi_u}^{-1}(j)} \in R_{10}(u), v+\sum_{j\in S}e_{{\pi_v}^{-1}(j)} \in R_{10}(v)\}$, then by Corollary \ref{bothways} and Lemma \ref{krigid}, $|C| \ge \binom{n}{10} -2en^9s $. We now split into three cases depending on the distance between $u$ and $v$. In each case we define a subset $C' \subseteq C$, which we will exploit later.

\begin{itemize}
\item[Case A:] \hspace{15pt} $u = v + e_s+e_t$. In this instance, let
\[
C' = \{S \in C :  (\pi_u^{-1}(S) \cup \pi_v^{-1}(S)) \cap \{s,t\} = \emptyset\}.
\]
Then $|C'| \ge \binom{n}{10} - O(n^9 s )$, and if $a \in \{u+\sum_{j\in S}e_{\pi_u^{-1}(j)}: S \in C'\}$ and $ b\in \{v+\sum_{j\in S}e_{\pi_v^{-1}(j)}: S \in C'\}$ then $a$ and $b$ are at an even distance at least two from each other.

\item[Case B:] \hspace{15pt} $u = v + e_s + e_t + e_r$. In this instance, let
\[
C' = \{S \in C :  (\pi_u^{-1}(S) \cup \pi_v^{-1}(S)) \cap \{s,t,r\} = \emptyset\},
\]
so $|C'| \ge \binom{n}{10} - O(n^9 s )$. If $a \in \{u+\sum_{j\in S}e_{\pi_u^{-1}(j)}: S \in C'\}$, then there may be a unique vertex in $\{v+\sum_{j\in S}e_{\pi_v^{-1}(j)}: S \in C'\}$ at distance three from $a$. In this case, let $b_a$ be this vertex and otherwise let $b_a$ be an arbitrary vertex at distance $3$ from $a$. If $a \in \{u+\sum_{j\in S}e_{\pi_u^{-1}(j)} : S \in C'\}$ and $b \in \{v+\sum_{j\in S}e_{\pi_v^{-1}(j)} : S \in C'\} \setminus \{b_a\}$, then the distance between $a$ and $b$ in the hypercube is at least $5$ (as the distance between them is odd and greater than $3$).

\item[Case C:] \hspace{15pt} $u$ and $v$ are at distance at least four from each other. In this instance, let $s,t,r,y$ be such that the distance between $u+e_s+e_t+e_r+e_y$ and $v$ is four less than the distance between $u$ and $v$. Then let
\[
C' = \{S \in C :  (\pi_u^{-1}(S) \cup \pi_v^{-1}(S)) \cap \{s,t,r,y\} = \emptyset\}.
\]
Then $|C'| \ge \binom{n}{10} - O(n^9 s )$, and if $a \in \{u+\sum_{j\in S}e_{\pi_u^{-1}(j)}: S \in C'\}$ and $ b\in \{v+\sum_{j\in S}e_{\pi_v^{-1}(j)}: S \in C'\}$ then $a$ and $b$ are at a distance at least four from each other.
\end{itemize}

We now come to fixing our substructures. Let $\delta, \varepsilon > 0$ be such that $2\varepsilon \delta < \frac{1}{10!}$ and choose sets $(A_s)_{s \in J}$ (with $|A_s| = \left\lceil \delta n^4 \right\rceil$ for each $s$, and $|J| = \left\lceil\varepsilon n^6\right\rceil$) as in Proposition \ref{greedycovering}. Note that $|\bigcup_{s \le \lceil \varepsilon n \rceil} A_s| \ge \delta \varepsilon n^{10}$ and so $|(\bigcup_{s \le \lceil \varepsilon n \rceil} A_s) \cap C'| \ge \delta \varepsilon n^{10} - O(n^9s )$. By the pigeonhole principle there exists a $j \in J$ such that $|A_{j} \cap C'| \ge \delta n^4 - O(n^3 s )$. Let $C'' = A_{j} \cap C'$. This approach of appealing to Lemma \ref{greedycovering} may seem unnecessary, but is important as it reduces the number of substructures we have to consider, in turn helping the union bound we take later.

We now give explicit events detailing how the colourings of our substructure have to ``fit in" with one another. Roughly speaking, for adjacent vertices $y$ and $z,$ we consider that the first neighbourhood of $y$ is the first neighbourhood of the neighbourhood of $z$. For all vertices $w \in V(Q_n)$, let
\[
\psi(w) = \sum_{x \in \Gamma(w)}\chi(x) - n(1-p),
\]
and then let
\[
\Psi(w) = \{\psi(x) : x \in \Gamma(w)\}.
\]
Recall that $\chi_f^{(2)}(f(w)) \cong \chi^{(2)}(w)$ for all $w \in V(Q_n)$. If $f(u)f(v) \in E(Q_n)$, then \eqref{star} gives
\[
\chi_f^{(2)}\left(f(u)+ \sum_{\ell \in S}e_{\ell}\right) \cong \chi^{(2)}\left(u+ \sum_{\ell \in S}e_{\pi_u^{-1}(\ell)}\right)
\]
and
\[
\chi_f^{(2)}\left(f(v)+ \sum_{\ell \in S}e_{\ell}\right) \cong \chi^{(2)}\left(v + \sum_{\ell\in S}e_{\pi_v^{-1}(\ell)}\right)
\]
for all $S \in C''$. This means that $\psi(u+ \sum_{\ell\in S}e_{\pi_u^{-1}(\ell)}) \in \Psi(v+ \sum_{\ell\in S}e_{\pi_v^{-1}(\ell)})$ for all $S \in C''$. For permutations $\pi_1,\pi_2$ and $S \subseteq [n]^{(10)}$, let $B^{\pi_1,\pi_2}_S$ be the event
	\begin{align}
		B^{\pi_1,\pi_2}_S = \left\{\psi\left(u+ \sum_{\ell\in S}e_{\pi_1(\ell)}\right) \in \Psi\left(v+ \sum_{\ell\in S}e_{\pi_2(\ell)}\right)\right\}. \nonumber
	\end{align}
Note that if $f(u)f(v) \in E(Q_n)$, then $B^{\pi_u^{-1},\pi_v^{-1}}_S$ occurs for all $S \in C''$.

Considering $\chi$ as a fixed colouring, given $j \in J$ and a pair of permutations $\pi_1,\pi_2$, we say that a subset $C'' \subseteq A_j$ of size $\delta n^4 - O(n^3 s)$ is a \em $(j,\pi_1,\pi_2)$-tester \em if $j, \pi_1,\pi_2,C''$ satisfy the properties outlined in Case A, Case B, or Case C as appropriate. Let $T_j(\pi_1,\pi_2)$ be the set of $(j,\pi_1,\pi_2)$-testers. If $f(u)f(v) \in E(Q_n)$ then there is a $j \in J$, pair of permutations $\pi_1,\pi_2$, and $C'' \in T_j(\pi_1,\pi_2)$ such that $B^{\pi_1^{-1},\pi_2^{-1}}_S$ occurs for all $S \in C''$.

We can then bound the probability that there exists an $f\in \mathrm{Diag}_{K's }$ for which $f(u)f(v) \in E(Q_n)$ and $f \in \mathrm{Isom}^{(2)}(\chi)$ by
	\begin{align}
		&\bP\left[\exists f\in \mathrm{Isom}^{(2)}(\chi) \ \mathrm{s.t.} \ f \in \mathrm{Diag}_{K's }, f(u)f(v) \in E(Q_n)\right] \nonumber \\
		&\qquad \le \bP\left[\bigcup_{\pi_1,\pi_2}\bigcup_{j\in J}\bigcup_{C'' \in T_j(\pi_1,\pi_2)}\bigcap_{S \in C''}B^{\pi_1,\pi_2}_S\right] \nonumber \\
		&\qquad \le \sum_{\pi_1,\pi_2}\sum_{j\in J}\sum_{C''\in T_j(\pi_1,\pi_2)}\bP\left[\bigcap_{S \in C''}B^{\pi_1,\pi_2}_S\right]. \nonumber
	\end{align}
Note that we have $\exp\left\{O\left(n\log n\right)\right\}$ choices for the permutations $\pi_1$ and $\pi_2$. We then have $|J|=O\left(n^6\right)$ choices for $j \in J.$ Finally, note that $T_j(\pi_1,\pi_2) \subseteq A_j^{\left(|A_j| - O\left(n^3s\right)\right)}$, so that there are at most $\binom{\delta n^4}{O\left(n^3s\right)} = \exp\left\{O\left(n^3s\log n\right)\right\}$ choices for $C'' \in T_j\left(\pi_1,\pi_2\right).$ Therefore, if we found a uniform upper bound $D$ for $\bP\left[\bigcap_{S \in C''}B^{\pi_1,\pi_2}_S\right],$ we would have
	\begin{align}
		\bP\left[\exists f\in \mathrm{Isom}^{(2)}(\chi) \ \mathrm{s.t.} \ f \in \mathrm{Diag}_{K's }, f(u)f(v) \in E(Q_n)\right] \le D\exp\left\{O\left(n^3s\log n\right)\right\}. \label{unionb1}
	\end{align}

Note that for each $w \in V(Q_n)$, $\psi(w)$ is determined by $(\chi(x))_{x \in \Gamma(w)}$, and $\Psi(w)$ is determined by $(\chi(x))_{x \in \Gamma^2(w) \cup \left\{w\right\}}$. Since the sets in $C''$ are all at distance at least $6$ from each other, $((\chi(x))_{x \in \Gamma(u+ \sum_{i\in S}e_{\pi_1(i)})})_{S \in C''}$ is a family of disjoint sets of random variables. This means that $(\psi(u+ \sum_{j\in S}e_{\pi_1(j)}))_{S \in C''}$ is a family of independent identically distributed random variables. Similarly, $(\Psi(v+ \sum_{j\in S}e_{\pi_2(j)}))_{S \in C''}$ is a family of independent identically distributed random variables.

We now come to finding our uniform upper bound $D$. We again have to split this up into the three cases. Case B is the hardest and the work covering this case also caters for Case A and Case C.

\begin{itemize}
\item[Case A:] \hspace{15pt} Suppose that $C''$ satisfies the properties outlined in Case A. Since all vertices $a \in \left\{u+\sum_{j\in S}e_{\pi_1(j)}: S \in C''\right\}$ and $ b\in \left\{v+\sum_{j\in S}e_{\pi_2(j)}: S \in C''\right\}$ are an even distance at least $2$ from each other, $\Gamma(a)$ and $\Gamma^2(b)\cup\left\{b\right\}$ do not intersect. Therefore $(\psi(u+ \sum_{j\in S}e_{\pi_1(j)}))_{S \in C''}$ and $(\Psi(v+ \sum_{j\in S}e_{\pi_2(j)}))_{S \in C''}$ are independent families of random variables and so, picking an arbitrary $S_0 \in C''$, 
	\begin{align}
		\bP\left[\bigcap_{S \in C''}B^{\pi_1,\pi_2}_S\right] &= \bP\left[B^{\pi_1,\pi_2}_{S_0}\right]^{|C''|}. \label{interA}
	\end{align}
	
\item[Case C:] \hspace{15pt} Suppose that $C''$ satisfies the properties outlined in Case C. Since all vertices $a \in \left\{u+\sum_{j\in S}e_{\pi_1(j)}: S \in C''\right\}$ and $ b\in \left\{v+\sum_{j\in S}e_{\pi_2(j)}: S \in C''\right\}$ are at distance at least $4$ from each other, $\Gamma(a)$ and $\Gamma^2(b)\cup\left\{b\right\}$ do not intersect. We can then follow the line of argument as in Case A, and \eqref{interA} again holds.

\item[Case B:] \hspace{15pt} Suppose that $C''$ satisfies the properties outlined in Case B. For each $a \in \{u+\sum_{j\in S}e_{\pi_1(j)}: S \in C''\}$, let $\psi'(a) = \sum_{w \in \Gamma(a) \setminus \Gamma^2(b_a)}\chi(w) - (n-3)(1-p)$. Then as in the previous cases, $(\psi'(u+ \sum_{j\in S}e_{\pi_1(j)}))_{S \in C''}$ and $(\Psi(v+ \sum_{j\in S}e_{\pi_2(j)}))_{S \in C''}$ are independent families of random variables. Define the events $\Lambda^{\pi_1,\pi_2}_S$ by 
\[
\Lambda^{\pi_1,\pi_2}_S = \left\{\psi'\left(u+ \sum_{j\in S}e_{\pi_1\left(j\right)}\right) \in \Psi\left(v+ \sum_{j\in S}e_{\pi_2\left(j\right)}\right)+[-3,3]\right\}.
\]
Since $|\Gamma(a) \cap \Gamma^2(b_a)| = 3$, we have $B^{\pi_1,\pi_2}_S \subseteq \Lambda^{\pi_1,\pi_2}_S$. Then picking an arbitrary $S_0 \in C''$, we obtain
	\begin{align}
		\bP\left[\bigcap_{S \in C''}B^{\pi_1,\pi_2}_S\right]&\le\bP\left[\bigcap_{S \in C''}\Lambda^{\pi_1,\pi_2}_S\right] \nonumber \\
		&= \bP\left[\Lambda^{\pi_1,\pi_2}_{S_0}\right]^{|C''|}.\label{interB}
	\end{align}
Note that, in fact, in cases A and C, for any $a \in \left\{u+\sum_{j\in S}e_{\pi_1(j)}: S \in C''\right\}$ we could define $b_a$ to be an arbitrary vertex at distance $3$ from $a$. Then, \eqref{interB} is in fact an upper bound in all three cases, hence we now focus on bounding that expression.

\qquad Let $x = u +  \sum_{j\in S_0}e_{\pi_1\left(j\right)}$ and $y = v+ \sum_{j\in S_0}e_{\pi_2\left(j\right)}$. To bound below the probability of $(\Lambda^{\pi_1,\pi_2}_{S_0})^C$, we condition on the value of $\psi'(x)$ and then consider whether $\psi\left(z\right) - \psi'(x)  \in[-3,3]$ for any $z \in \Gamma(y)$. Note that we will just be considering atypical values of $\psi'(x)$. This means that our lower bound is very close to $0,$ but since we will be considering a large intersection of independent events, it suffices to give a lower bound that is not too close to $0$. Let
\[
c \in \left(\sqrt{\frac{5-4 \varepsilon}{3+4\varepsilon} (1-p)},\sqrt{\frac{5}{3} (1-p)}\right),
\]
so that $\frac{1}{6}+\frac{c^2}{2(1-p)} < 1$ and
\begin{align*}
\left( \frac{3}{4} + \varepsilon \right) \left( \frac{1}{2}+\frac{c^2}{2(1-p)} \right) & > \frac{3+4\varepsilon}{4} \left( \frac{1}{2}+\frac{\frac{5-4 \varepsilon}{3+4\varepsilon} (1-p)}{2(1-p)} \right) \\
& = \frac{3+4\varepsilon}{8} \cdot \frac{3+4\varepsilon+5-4 \varepsilon}{3+4\varepsilon} = 1,
\end{align*}
and then let $M = c\left(np \log (np) \right)^{\frac{1}{2}}$. Taking a union bound gives
	\begin{align}
		\bP\left[(\Lambda^{\pi_1,\pi_2}_{S_0})^C\right] &\ge \bP\left[\psi'\left(x\right) \ge M \mbox{ and} \left(\Lambda^{\pi_1,\pi_2}_{S_0}\right)^C\right] \nonumber \\
		&\ge \bP\left[\psi'\left(x\right) \ge M\right]\left(1-\sum_{z \in \Gamma(y)} \bP\left[\psi\left(z\right) - \psi'\left(x\right)  \in[-3,3]|\psi'\left(x\right) \ge M\right]\right) \nonumber \\
		&\ge \left(1-n\bP\left[\psi\left(z\right) - M \in [-3,3]\right]\right)\bP\left[\psi'\left(x\right) \ge M\right], \nonumber
	\end{align}
where the last inequality follows from the fact that $\psi(x)$ is a normalised binomial random variable with mean $0$. Since the same applies to $\psi'$, and recalling that $(3/4 + \varepsilon)(\frac{1+c^2}{2}) > 1$ and $p \ge n^{-1/4+\varepsilon}$, we therefore appeal to Lemma \ref{bounds} to get
	\begin{align*}
		\bP\left[(\Lambda^{\pi_1,\pi_2}_{S_0})^C\right] &\ge \left(1-n\Theta\left(\left(np\right)^{-\left(\frac{1}{2}+\frac{c^2}{2(1-p)}\right)}\right)\right)\Omega\left(\left(np\right)^{-\left(\frac{1}{6}+\frac{c^2}{2\left(1-p\right)}\right)}\right) \\
		&\ge \left(1-n\Theta\left(n^{-\left(3/4 +\varepsilon\right)\left(\frac{1}{2}+\frac{c^2}{2(1-p)}\right)}\right)\right)\Omega\left(\left(np\right)^{-\left(\frac{1}{6}+\frac{c^2}{2\left(1-p\right)}\right)}\right) \\
		&= \Omega\left(\left(np\right)^{-\left(\frac{1}{6}+\frac{c^2}{2(1-p)}\right)}\right).
	\end{align*}
Let $\Delta = \frac{\log np}{\log n} \in [3/4+\varepsilon,1)$, so that $np = n^{\Delta}$. We may express the above inequality as
	\begin{align*}
		\bP\left[\Lambda^{\pi_1,\pi_2}_{S_0}\right] = 1 - \Omega\left(n^{-\Delta\left(\frac{1}{6}+\frac{c^2}{2(1-p)}\right)}\right).
	\end{align*}
Putting this into \eqref{interB} we see
	\begin{align*}
		\bP\left[\bigcap_{S \in C''}B^{\pi_1,\pi_2}_S\right]&\le \left(1 - \Omega\left(n^{-\Delta\left(\frac{1}{6}+\frac{c^2}{2(1-p)}\right)}\right)\right)^{\delta n^4 -O\left(n^3 s \right)} \\
		&= \exp\left\{-\Omega\left(n^{4 - \Delta\left(\frac{1}{6}+\frac{c^2}{2(1-p)}\right)}\right)\right\}.
	\end{align*}
\end{itemize}

We have found our uniform upper bound $D$ and so \eqref{unionb1} gives
	\begin{align*}
		&\bP\left[\exists f\in \mathrm{Isom}^{(2)}(\chi) \ \mathrm{s.t.} \ f \in \mathrm{Diag}_{K'\log n }, f(u)f(v) \in E(Q_n)\right] \\
		&\qquad = \exp\left\{O\left(n^3 s\log n\right) - \Omega\left(n^{4 - \Delta \left(\frac{1}{6}+\frac{c^2}{2(1-p)}\right)}\right)\right\}.
	\end{align*}

Recall that $s = p^{-1}\log n = n^{1-\Delta} \log n$ and so
	\begin{align*}
		&\bP\left[\exists f\in \mathrm{Isom}^{(2)}(\chi) \ \mathrm{s.t.} \ f \in \mathrm{Diag}_{K'\log n }, f(u)f(v) \in E(Q_n)\right] \\
		&\qquad= \exp\left\{O\left(n^{4-\Delta} \log^2 n\right) - \Omega\left(n^{4 - \Delta\left(\frac{1}{6}+\frac{c^2}{2(1-p)}\right)}\right)\right\}.
	\end{align*}

We chose $c$ so that $\frac{1}{6}+\frac{c^2}{2(1-p)} < 1$ and so $n^{4-\Delta} \log^2 n = o \left(n^{4 - \Delta\left(\frac{1}{6}+\frac{c^2}{2(1-p)}\right)} \right)$. As we already observed, for $\chi \circ f^{-1} \not\cong \chi$, there must be a pair of non-adjacent vertices $u$ and $v$ such that $f(u)f(v) \in E(Q_n)$. We have fewer than $2^{2n}$ choices for $u$ and $v$, and so taking a union bound gives
	\begin{align*}
		&\bP\left[\exists f \in \mathrm{Isom}^{(2)}(\chi) \ \mathrm{s.t.} \ f\in \mathrm{Diag}_{K'\log n },\chi \circ f^{-1} \not\cong \chi\right] \\
		&\qquad = \exp\left\{O(n)+O\left(n^{4-\Delta} \log^2 n\right) - \Omega\left(n^{4 - \Delta\left(\frac{1}{6}+\frac{c^2}{2(1-p)}\right)}\right)\right\} \\
		&\qquad = o(1).
	\end{align*}
Finally, we can conclude that $\bP\left[\chi \ \mathrm{is} \mbox{ $2$-indistinguishable}\right] = o(1)$.
\end{proof}

\section{Proof of Theorem \ref{1ball}}\label{1ballsec}
As with Theorem \ref{main}, we prove Theorem \ref{1ball} by combining some of the probabilistic and structural results already proven. We start off with a lemma to discount bijections which map large parts of neighbourhoods to neighbourhoods.

\begin{lem}\label{matching}
For any $K>0$, there exists a constant $C = C(K)$ such that the following holds: Let $q(n) \ge n^{2+C\log^{-\frac{1}{2}} n}$, and let $\chi$ be a random $q$-colouring of the hypercube $Q_n$. Then with high probability, there does not exist a bijection $f \in \mathrm{Local}_{n(1-K\log^{-\frac{1}{2}} n)}$ and a pair of non-adjacent vertices $u,v$ such that $f \in \mathrm{Isom}^{(1)}(\chi)$ and $f(u)f(v) \in E(Q_n)$.
\end{lem}

It will be useful in the proof to introduce the following piece of notation: 
\begin{defn}\label{selfdual}
For a $s(n)$-approximately local bijection $f$, we say it is \em self-dual \em if it is its own dual and this dual is unique, i.e. if $f_{\star} = f$.
\end{defn}

For a natural number $s = s(n)$, let $\mathrm{Self}_s$ be the set of self-dual bijections in $\mathrm{Local}_s$, i.e. let $\mathrm{Self}_s = \left\{f \in \mathrm{Local}_s : f_{\star} = f\right\}$.

\begin{proof}[Proof of Lemma \ref{matching}]
Let $K>0$, let $C > 0$ be a constant to be defined later, and let $q(n) \ge n^{2+C\log^{-\frac{1}{2}} n}$. For ease of notation, let $M = n(1-K\log^{-\frac{1}{2}} n)$. Let $\chi$ be a random $q$-colouring of the hypercube $Q_n$. First suppose that there exists a bijection $f \in \mathrm{Local}_M\setminus \mathrm{Self}_M$ such that $f \in \mathrm{Isom}^{(1)}(\chi)$. Let $f_{\star}$ be a dual of $f$ (note that since $M > n/2$, there may not be a unique dual).

Pick $w \in V(Q_n)$ such that $f_{\star}(w) \neq f(w)$. Then $|\Gamma(w) \cap f^{-1}(\Gamma(f_{\star}(w)))| \ge Kn\log^{-\frac{1}{2}} n$, since $f \in \mathrm{Local}_M$, and so $d(\chi^{(1)}(f^{-1}(f_{\star}(w))),\chi^{(1)}(w)) \le n(1-K\log^{-\frac{1}{2}} n)$. Since we assumed that $f(w) \neq f_{\star}(w)$, we see that there must exist some $x \neq y \in V(Q_n)$ such that $d(\chi^{(1)}(x),\chi^{(1)}(y)) \le n(1-K\log^{-\frac{1}{2}} n)$. By Lemma \ref{1strongunique}, the probability of this occurring is $o(1)$ and so
	\begin{align*}
		\bP\left[\exists f \in  \mathrm{Isom}^{(2)}(\chi)\ \mathrm{s.t.} \ f \in \mathrm{Local}_M \setminus \mathrm{Self}_M\right] = o(1).
	\end{align*}

Pick two non-adjacent vertices $u$ and $v$. Suppose that there exists a bijection $f \in \mathrm{Self}_M$ such that $f \in \mathrm{Isom}^{(1)}(\chi)$ and $f(u)f(v) \in V(Q_n)$, and let
	\begin{align*}
		U = \left\{w \in \Gamma(u) : f(w) \in \Gamma(f(u)), d(w,v) \neq 2\right\}.
	\end{align*}
Recall that $u$ and $v$ are non-adjacent vertices, and so $|\left\{w \in \Gamma(u) : d(w,v) = 2\right\}| \le 3$. Also consider that $f \in \mathrm{Self}_M$ and so $|U| \ge n-M-3 = Kn\log^{-\frac{1}{2}} n-3$.

For each $w \in U$, $f(w)$ is at distance $2$ from $f(v)$ in the hypercube and so there is a distinct $i_w \in [n]$ such that
	\begin{align*}
		\Gamma\left(f(v)\right) \cap \Gamma\left(f(w)\right) = \left\{f(u),f(v)+e_{i_w}\right\}.
	\end{align*}
Recall that $\chi^{(1)}(w)\cong \chi_f^{(1)}(f(w))$ and so $\chi_f\left(f(v) + e_{i_w}\right) \in \chi\left(\Gamma(w)\setminus\{u\}\right).$ Let $Y = \Gamma^2(u) \setminus \Gamma(v)$. For each $w \in U,$ $\Gamma(w)\setminus\{u\} \subseteq Y$ since $w \in \Gamma(u)$ and $d(v,w) \neq 2$. Therefore $\chi_f\left(f(v) + e_{i_w}\right) \in \chi(Y)$ for all $w \in U$.

Since $\chi_f^{(1)}(f(v)) \cong \chi^{(1)}(v)$, there exists a permutation $\pi$ such that $\chi(v + e_{\pi(i)}) = \chi_f(f(v)+e_i)$ for all $i\in [n]$. But then $\chi(v + e_{\pi(i_w)}) = \chi_f\left(f(v) + e_{i_w}\right) \in \chi(Y)$ for all $w \in U$. Then there exists a set $T_U \subseteq [n]$ of size $\frac{K}{2}n\log^{-\frac{1}{2}} n$ such that $\chi(v+e_i) \in \chi(Y)$ for all $i\in T_U$. Therefore
	\begin{align}
		&\bP\left[\exists f \in \mathrm{Isom}^{(1)}(\chi) \ \mathrm{s.t.} \ f \in \mathrm{Self}_M, f(u)f(v) \in E(Q_n)\right] \nonumber \\
		&\le \bP\left[\exists T_U \in [n]^{\left(\frac{K}{2}n\log^{-\frac{1}{2}} n\right)} \ \mathrm{s.t.} \ \forall i\in T_U \ \chi(v+e_i) \in \chi(Y)\right]. \label{matchingspot1}
	\end{align}
Now $(\chi(v+e_i))_{i \in [n]}$ and $(\chi(x))_{x \in Y}$ are independent families of independent $\mathrm{Unif}([q])$ random variables and so for an arbitrary $T \in [n]^{(\frac{K}{2}n\log^{-\frac{1}{2}} n)}$
	\begin{align*}
		\bP\left[\forall i\in T \ \chi(v+e_i) \in \chi(Y) \ | \ \chi(Y)\right\} &= \prod_{i\in T} \bP\left[\chi(v+e_i) \in \chi(Y) \ | \ \chi(Y)\right] \\
		&= \left(\frac{|\chi(Y)|}{q}\right)^{\frac{K}{2}n\log^{-\frac{1}{2}} n} \\
		&\le \left(\frac{n^2}{q}\right)^{\frac{K}{2}n\log^{-\frac{1}{2}} n}.
	\end{align*}
We can take an expectation over $\chi(Y)$ to get
	\begin{align*}
		\bP\left[\forall i\in T \ \chi(v+e_i) \in \chi(Y)\right] \le  \left(\frac{n^2}{q}\right)^{\frac{K}{2}n\log^{-\frac{1}{2}} n}.
	\end{align*}
We can then apply a union bound to \eqref{matchingspot1} to get the following bound
	\begin{align*}
		\bP\left[\exists f \in \mathrm{Isom}^{(1)}(\chi) \ \mathrm{s.t.} \ f \in \mathrm{Self}_M, f(u)f(v) \in E(Q_n)\right] &\le \binom{n}{\frac{K}{2}n\log^{-\frac{1}{2}} n}\left(\frac{n^2}{q}\right)^{\frac{K}{2}n\log^{-\frac{1}{2}} n} \\
		&\le \left(\frac{en}{\frac{K}{2}n\log^{-\frac{1}{2}} n}\right)^{\frac{K}{2}n\log^{-\frac{1}{2}} n} \left(\frac{n^2}{q}\right)^{\frac{K}{2}n\log^{-\frac{1}{2}} n} \\
		&= \left(\frac{2en^2\log^{1/2}n}{Kq}\right)^{\frac{K}{2}n\log^{-\frac{1}{2}} n}. \\
	\end{align*}
Define $D = \frac{K}{2} \log \left(\frac{2e}{K}\right) $, a constant depending on $K$. Recall that $q \ge n^{2+C\log^{-1/2}n}$, and so the bound above is at most 
	\begin{align*}
		\left(\frac{2e}{K} n^{-C\log^{-1/2}n}\log^{1/2}n\right)^{\frac{K}{2}n\log^{-\frac{1}{2}} n} = \exp\left\{Dn\log^{-1/2}n - \frac{CK}{2}n + \frac{K}{4}n\log^{-1/2}n\log\left(\log n\right) \right\}.
	\end{align*}
Taking $C$ sufficiently large we get
	\begin{align*}
		\bP\left[\exists f \in \mathrm{Isom}^{(1)}(\chi) \ \mathrm{s.t.} \ f \in \mathrm{Self}_M, f(u)f(v) \in E(Q_n)\right] \le e^{-\frac{CK}{4}n}.
	\end{align*}

We have fewer than $2^{2n}$ choices for non-adjacent vertices $u$ and $v$ and so by a union bound,
	\begin{align*}
		\bP\left[\exists uv \not\in E(Q_n), f \in \mathrm{Isom}^{(1)}(\chi) \ \mathrm{s.t.} \ f \in \mathrm{Self}_M, f(u)f(v) \in E(Q_n)\right] \le 2^{2n}e^{-\frac{CK}{4}n}.
	\end{align*}
This upper bound is $o(1)$ if $C$ is large enough and so
	\begin{align*}
		\bP\left[\exists uv \not\in E(Q_n), f \in \mathrm{Isom}^{(1)}(\chi) \ \mathrm{s.t.} f \in \mathrm{Local}_M, f(u)f(v) \in E(Q_n)\right] = o(1).
	\end{align*}
\end{proof}

We are now in a position to prove Theorem \ref{1ball}.

\begin{proof}[Proof of Theorem \ref{1ball}]
Let $K,K_1,K_2 > 0$ be constants to be defined later, and then let $\varepsilon (n) = \frac{1}{2} - K_2\log^{-\frac{1}{2}} n$ and $q \ge K_1n^{2 + 2K_2\log^{-\frac{1}{2}} n}$. Let $\chi$ be a random $q$-colouring of the hypercube $Q_n$. By Lemma \ref{1scott}, if $K_1$ is sufficiently large then
	\begin{align*}
		\bP\left[\exists f \in \mathrm{Isom}^{(1)}(\chi) \ \mathrm{s.t.} \ f^{-1} \not\in \mathrm{Cluster}_{\varepsilon(n)n^2}^1 \right] = o(1),
	\end{align*}
and so
	\begin{align*}
		\bP\left[\chi \ \mathrm{is} \ \mbox{$1$-indist.}\right] = \bP\left[\exists f \in \mathrm{Isom}^{(1)}(\chi) \ \mathrm{s.t.} \ f^{-1}\in \mathrm{Cluster}^1_{\varepsilon(n)n^2}, \chi \circ f^{-1} \not\cong \chi\right] + o(1).
	\end{align*}
	
Now suppose that there exists a bijection $f \in \mathrm{Isom}^{(1)}(\chi)$ with $f^{-1} \in \mathrm{Cluster}^1_{\varepsilon(n)n^2} \setminus \mathrm{Mono}_{\varepsilon(n)n^2}^{Kn\log^{-1} n}$. Since $f^{-1} \not\in \mathrm{Mono}_{\varepsilon(n)n^2}^{Kn\log^{-1} n}$ there must exist vertices $v,w_1,w_2$ such that $w_1 \neq w_2$ and $|f^{-1}(\Gamma(v))\cap \Gamma(w_i)| > Kn\log n ^{-1}$ for $i = 1,2$. Note that $|f^{-1}(\Gamma(v))\cap \Gamma(w_i)| > Kn\log n ^{-1}$ implies that $d(\chi_f^{(1)}(v),\chi^{(1)}(w_i)) \le n - K\frac{n}{\log n }$ for $i=1,2$. Recall that $\chi_f^{(1)}(v) = \chi^{(1)}(f^{-1}(v))$ and so $d(\chi^{(1)}(f^{-1}(v)),\chi^{(1)}(w_i)) \le n - K\frac{n}{\log n }$ for $i=1,2$. It cannot be the case that $w_1=w_2 = f^{-1}(v)$ and so we have found two vertices $u\neq x$ such that $d(\chi^{(1)}(u),\chi^{(1)}(x)) \le n-K\frac{n}{\log n }$. By Lemma \ref{1strongunique}, if $K$ is sufficiently large, this occurs with probability $o(1)$ and so
	\begin{align}
		\bP\left[\chi \ \mathrm{is} \ \mbox{$1$-indist.}\right] = \bP\left[\exists f \in \mathrm{Isom}^{(1)}(\chi) \ \mathrm{s.t.} \ f^{-1}\in \mathrm{Mono}_{\varepsilon(n)n^2}^{Kn\log^{-1} n}, \chi \circ f^{-1} \not\cong \chi\right] + o(1). \nonumber
	\end{align}

In a similar fashion, we could show that there cannot exist vertices $v_1,v_2,w$ such that $v_1 \neq v_2$ and $|f^{-1}(\Gamma(v_i)) \cap	\Gamma(w)| > K n \log^{-1}n$ for $i = 1,2$. Now, recall that by Corollary \ref{1ftoj}
	\begin{align*}
		\mathrm{Mono}_{\varepsilon(n)n^2}^{Kn\log^{-1} n} \subseteq \mathrm{Local}_{y(n)},
	\end{align*}
where
	\begin{align*}
		y\left(n\right) &= n\left(1 - \left(1 - 2\varepsilon\left(n\right) - 14 \left(\frac{K}{\log n }\right)^{\frac{1}{2}}\right)^{\frac{1}{2}}\right) \\
		&= n\left(1 - \left(2K_2-14K^{\frac{1}{2}}\right)^{\frac{1}{2}}\log n ^{-\frac{1}{4}}\right).
	\end{align*}
So then if we take $K_2 > 8K^{\frac{1}{2}}$, 
	\begin{align*}
		y\left(n\right) \le n\left(1 - K^{\frac{1}{4}} \log^{-\frac{1}{4}} n\right),
	\end{align*}
and then since $\mathrm{Local}_R \subseteq \mathrm{Local}_T$ when $R \le T$, we see that
	\begin{align*}
		\mathrm{Mono}_{\varepsilon\left(n\right)n^2}^{Kn\log^{-1} n} \subseteq \mathrm{Local}_{n\left(1 - K^{\frac{1}{4}} \log^{-\frac{1}{4}} n\right)},
	\end{align*}
and any $f^{-1} \in \mathrm{Mono}_{\varepsilon\left(n\right)n^2}^{Kn\log^{-1} n}$ has a unique dual $g$.

Suppose that $g$ is not bijective. Then there exist vertices $v_1,v_2,w$ such that $v_1 \neq v_2$ and $|f^{-1}(\Gamma(v_i)) \cap \Gamma(w)| > K n \log^{-1}n$ for $i = 1,2$. By Lemma \ref{1strongunique} this cannot be the case, so $g$ must be bijective.

Since $f^{-1} \in \mathrm{Local}_{n\left(1 - K^{\frac{1}{4}} \log^{-\frac{1}{4}} n\right)}$ with bijective dual $g$, we may apply Lemma \ref{a.l.inverse} to get
	\begin{align*}
		&\bP\left[\chi \ \mathrm{is} \ \mbox{$1$-indist.}\right] = \bP\left[f \in \mathrm{Isom}^{(1)}(\chi) \ \mathrm{s.t.} \ f^{-1}\in \mathrm{Mono}_{\varepsilon(n)n^2}^{Kn\log^{-1} n}, \chi \circ f^{-1} \not\cong \chi\right] + o(1) \\
		& \qquad \le \bP\left[\exists f \in \mathrm{Isom}^{\left(1\right)}\left(\chi\right) \ \mathrm{s.t.} \ f \in \mathrm{Local}_{n\left(1- K^{\frac{1}{4}} \log^{-\frac{1}{4}} n\right)}, \chi \circ f^{-1} \not\cong \chi\right] + o\left(1\right).
	\end{align*}
Finally, since $\log^{-\frac{1}{2}} n = o\left(\log n ^{-\frac{1}{4}}\right)$, we may apply Lemma \ref{matching} to conclude that
	\begin{align*}
		\bP\left[\exists f \in \mathrm{Isom}^{\left(1\right)}\left(\chi\right) \ \mathrm{s.t.} \ f\in \mathrm{Local}_{n\left(1 - K^{\frac{1}{4}} \log^{-\frac{1}{4}} n\right)}, \chi \circ f^{-1} \not\cong \chi\right] = o\left(1\right),
	\end{align*}
and so $\bP\left[\chi \ \mathrm{is} \ \mbox{$1$-indistinguishable}\right] = o\left(1\right)$.
\end{proof}

\section{Some Further Questions}\label{waffle}
In Theorem \ref{main}, we have the condition that $p\ge n^{-1/4+\epsilon}$.  How small can $p$ be taken here?  Is there a threshold function $\tau$ such that if ${p}/{\tau} \rightarrow \infty$, then a random $(p,1-p)$-colouring is $2$-distinguishable with high probabiilty, but the same is not true if $p = o(\tau)$? More generally, given a function $p$, how large must $r$ be so that a random $(p,1-p)$-colouring is $r$-distinguishable with high probability?

It is natural to ask whether Theorem \ref{main} extends to colourings with more colours.  We note the following corollary of Theorem \ref{main}.

\begin{cor}\label{BBC1966}
Let $\varepsilon > 0$ and let $\mu_n$ be a sequence of probability mass functions on the natural numbers for which $1-\mu_n(m) \ge n^{-1/4 + \varepsilon}$ uniformly over $m \in \bN$ for sufficiently large~$n$ (i.e. there is no single colour with probability mass too close to 1). Let $\chi$ be a random $\mu_n$-colouring of the hypercube $Q_n$. Then with high probability, $\chi$ is $2$-distinguishable.
\end{cor}

To prove this corollary, partition $\bN$ into two parts $A_n$ and $B_n$ for each $n$ so that $n^{-1/4+\varepsilon} \le \mu_n(A_n) \le \mu_n(B_n)$ for sufficiently large $n$. Then consider the colouring $\chi'$ where $\chi' = 0$ when $\chi \in A_n$ and $\chi'=1$ when $\chi \in B_n$. By Theorem \ref{main}, with high probability, we may reconstruct $\chi'$. From there we appeal to the uniqueness of local colourings of $2$--balls (see Lemma \ref{strongunique}) to recover $\chi$ from $\chi'$ with high probability.

It would be interesting to have better bounds on the values of $q$ for which a random $q$-colouring is $1$-distinguishable with high probability. We have an upper bound of the form $n^{2+o(1)}$ and a lower bound of form $\Omega(n)$; we expect $n^{1+o(1)}$ should be possible, and  Lemma \ref{1strongunique} shows that neighbourhoods are unique down to this range.  (It seems likely that this should be a monotone property, that is if a random $q$-colouring is $1$-distinguishable with high probability then the same be true for a random $(q+1)$-colouring, but we do not have a proof of this.)

Another interesting question is a different type of random jigsaw puzzle.
\begin{qu}
Let $q=q(n)$ be a positive integer, and let $V(Q_n)=S_1\cup\dots \cup S_q$ be a partition of the vertices of the cube into $q$ sets, chosen uniformly at random.
Suppose we are given each set $S_i$ up to an isometry.  When can the partition be reconstructed with high probability?
\end{qu}
An equivalent way to state this is the following: let $c$ be a random $q$-colouring of the vertices of $Q_n$, and suppose that $f:V(Q_n)\to V(Q_n)$ is a bijection such that, for every colour $k$, the restriction of $f$ to the vertices of colour $k$ is an isometry.  When is it almost surely the case that $f$ must be an isometry of the whole cube?  Of course, the interesting question is how large $q(n)$ can be.

Let us conclude by noting that there are other interesting questions about reconstructing colourings of the hypercube. For example, Keane and den Hollander \cite{KdH} asked when it is possible to reconstruct a colouring $c$ of a graph $G$ by observing $(c(X_n))_{n \in \bN}$, where $X_n$ is a random walk on the vertex set of $G$ (see also Benjamini and Kesten \cite{BK}). For the cube, not all colourings are reconstructible in this way, but for random colourings the problem is very much open (see Gross and Grupel \cite{GG} for the problem and discussion, and van Hintum \cite{Hintum} for further constructions).

\bibliographystyle{abbrv}

\newcommand{\SortNoop}[1]{}

\end{document}